\documentclass[letterpaper,reqno]{article}
\usepackage{amsmath}
\usepackage{amsfonts}
\usepackage{amssymb}
\usepackage{amstext}
\usepackage{amsbsy}
\usepackage{amsthm}
\usepackage{amscd}
\usepackage[all]{xy}

\newtheorem{theorem}{Theorem}[section]

\newtheorem{lem}[theorem]{Lemma}
\newtheorem{prop}[theorem]{Proposition}
\newtheorem{dfn}[theorem]{Definition}
\newtheorem{rem}[theorem]{Remark}

\begin{document}
\title{Note on a Geometric Isogeny of K3 Surfaces}
\author{Adrian Clingher
\thanks{
Department of Mathematics and Computer Science, University of Missouri-St. Louis, St. Louis  MO 63108. {\bf e-mail:} {\it clinghera@umsl.edu}
}
\and
Charles F. Doran\thanks{
Department of Mathematical and Statistical Sciences, University of Alberta. Edmonton AB T6G 2G1. {\bf e-mail:} {\it doran@math.ualberta.ca}
}
}
\date{}
\maketitle
\begin{center}
\abstract{
\noindent The paper establishes a correspondence relating two specific classes of complex algebraic K3 surfaces. The first class 
consists of K3 surfaces polarized by the rank-sixteen lattice ${\rm H} \oplus {\rm E}_7 \oplus {\rm E}_7 $. 
The second class consists of K3 surfaces obtained as minimal resolutions of double covers of the projective plane branched 
over a configuration of six lines. The correspondence underlies a geometric two-isogeny of K3 surfaces.  
}
\end{center}

\section{Geometric Two-Isogenies on K3 Surfaces}
\label{isogenydef}
Let ${\rm X}$ be an algebraic K3 surface defined over the field of complex numbers. A {\it Nikulin (or symplectic) involution} on ${\rm X}$ is an analytic 
automorphism of order two $\Phi \colon {\rm X} \rightarrow {\rm X} $ such that $ \Phi^*(\omega)=\omega$ 
for any holomorphic two-form $\omega$ on ${\rm X}$. This type of involution has many
interesting properties (see \cite{morrison, nikulin2}), amongst which the most important are: (a) the fixed locus of 
$\Phi$ consists of precisely eight distinct points, and (b) the surface ${\rm Y}$ obtained as the minimal resolution of the quotient ${\rm X}/\Phi$ is a K3 surface.
Equivalently, one can construct ${\rm Y}$ as follows. Blow up the eight fixed points on ${\rm X}$ obtaining a new surface
$\widetilde{{\rm X}}$. The Nikulin involution $\Phi$ extends to an involution $ \widetilde{\Phi}$ on $\widetilde{{\rm X}}$ which has as fixed locus
the disjoint union of the eight resulting exceptional curves. The quotient $\widetilde{{\rm X}} / \widetilde{\Phi}$ is smooth and recovers the
surface ${\rm Y}$ from above.
\par In the context of the above construction, one has a degree-two rational map ${\rm p}_{\Phi} \colon {\rm X} \dashrightarrow {\rm Y}$ with a branch locus given by 
eight disjoint rational curves (the even eight configuration in the sense of Mehran \cite{afsaneh}). In addition, there is a push-forward morphism (see \cite{shiodainose,morrison})
\begin{equation}
\label{morphh}
({\rm p}_{\Phi})_* \colon {\rm H}^2({\rm X}, \mathbb{Z}) \rightarrow {\rm H}_{{\rm Y}}
\end{equation}
mapping into the orthogonal complement in ${\rm H}^2({\rm Y}, \mathbb{Z})$ of the even eight curves. The
metamorphosis of the surface ${\rm X}$ into ${\rm Y}$ is referred to in the literature as the {\it Nikulin
construction}.
\par The most well-known class of Nikulin involutions is given by the {\it Shioda-Inose structures} \cite{shiodainose,morrison,nikulin2}.
These consist of Nikulin involutions that satisfy two additional requirements. The first condition asks for the surface ${\rm Y}$ to be Kummer.
The second requirement asserts that the morphism $(\ref{morphh})$ induces a Hodge isometry between the lattices of transcendental cocycles $ {\rm T}_{{\rm X}}(2) $ and ${\rm T}_{{\rm Y}} $. An effective criterion for a particular
K3 surface ${\rm X}$ to admit a Shioda-Inose structure was given by Morrison \cite{morrison}.
\par In this paper, we shall work with another class of Nikulin involutions: fiber-wise translations by a section of order two in a jacobian 
elliptic fibration. This class of involutions was discussed by Van Geemen and Sarti \cite{sarti1}. Let us be precise:
\begin{dfn}
\label{geemensartidef}
A {\em Van Geemen-Sarti involution} is an automorphism $\Phi_{{\rm X}} \colon {\rm X} \rightarrow {\rm X} $ for which there exists
a triple $(\varphi_{{\rm X}}, S_1, S_2)$ such that:
\begin{itemize}
\item [(a)] $\varphi_{{\rm X}} \colon {\rm X} \rightarrow {\mathbb P}^1$ is an elliptic fibration on ${\rm X}$,
\item [(b)] ${\rm S}_1$ and ${\rm S}_2$ are disjoint sections of $\varphi_{{\rm X}}$,
\item [(c)] ${\rm S}_2$ is an element of order two in the Mordell-Weil group ${\rm MW}(\varphi_{{\rm X}}, {\rm S}_1)$,
\item [(d)] $\Phi_{{\rm X}}$ is the involution obtained by extending the fiber-wise translations by ${\rm S}_2$ in the smooth fibers of
$\varphi_{{\rm X}}$ using the group structure with neutral element given by ${\rm S}_1$.
\end{itemize}
Under the above conditions, one says that the triple $(\varphi_{{\rm X}}, S_1, S_2)$ is {\em compatible} with the involution $\Phi_{{\rm X}}$.
\end{dfn}
\noindent Any given Van Geemen-Sarti involution is, in particular, a Nikulin involution. One can naturally
regard a Van Geemen-Sarti involution $\Phi_{{\rm X}}$ as a fiber-wise two-isogeny between the
original K3 surface ${\rm X}$ and the newly constructed K3 surface ${\rm Y}$. Since $\Phi_{{\rm X}}$ acts
as a translation by an element of order two in each of the
smooth fibers of $\varphi_{{\rm X}}$, there is a canonically induced elliptic fibration
$\varphi_{{\rm Y}} \colon {\rm Y} \rightarrow {\mathbb P}^1$. The new fibration $\varphi_{{\rm Y}}$ carries two special
sections, ${\rm S}'_1$ and ${\rm S}'_2$, as follows. The section ${\rm S}'_1$
is the image under the map ${\rm p}_{\Phi_{{\rm X}}}$ of the two sections ${\rm S}_1$ and ${\rm S}_2$ of $\varphi$. The section ${\rm S}'_2$
is the image under ${\rm p}_{\Phi_{{\rm X}}}$ of the divisor on ${\rm X}$ obtained by compactifying the curve obtained by taking the union of
remaining two order-two points in the smooth fibers of $\varphi_{{\rm X}}$. The two sections ${\rm S}'_1$ and ${\rm S}'_2$ are disjoint
and ${\rm S}'_2$ represents an element
of order two in the Mordell-Weil group ${\rm MW}(\varphi_{{\rm Y}}, {\rm S}'_1)$. Then, by standard results \cite{shioda2}, the 
fiber-wise translations by the order-two section ${\rm S}'_2$ extend to determine an involution
$ \Phi_{{\rm Y}} \colon {\rm Y} \rightarrow {\rm Y}$ which is a Van Geemen-Sarti
involution on ${\rm Y}$.
\par The same procedure applied initially to the involution $\Phi_{{\rm Y}}$ recovers the
K3 surface ${\rm X}$ together with the triple $(\varphi_{{\rm X}}, {\rm S}_1, {\rm S}_2)$ and the involution $\Phi_{{\rm X}}$. One has therefore the
following commutative diagram.
\begin{equation}
\label{chartt1}
\xymatrix 
{
{\rm Y} \ar @(dl,ul) _{\Phi_{{\rm Y}}} \ar [dr] _{\varphi_{{\rm Y}}} \ar @/_0.5pc/ @{-->} _{{\rm p}_{\Phi_{{\rm Y}}}} [rr]
&
& {\rm X} \ar @(dr,ur) ^{\Phi_{{\rm X}}} \ar [dl] ^{\varphi_{{\rm X}}} \ar @/_0.5pc/ @{-->} _{{\rm p}_{\Phi_{{\rm X}}}} [ll] \\
& \mathbb{P}^1 & \\
}
\end{equation}
The rational maps ${\rm p}_{\Phi_{{\rm X}}}$ and ${\rm p}_{\Phi_{{\rm Y}}}$ are of degree two. Hence,
$({\rm p}_{\Phi_{{\rm X}}}, {\rm p}_{\Phi_{{\rm Y}}})$ can be be seen as forming {\it a pair of dual two-isogenies}\footnote{
The rational maps ${\rm p}_{\Phi_{{\rm X}}}$ and ${\rm p}_{\Phi_{{\rm Y}}}$ are not isogenies in the traditional sense 
(finite and etale morphism). The first author would like to thank Mohan Kumar for pointing out this fact.   
} between the
surfaces ${\rm X}$ and ${\rm Y}$.
%
%
%
%
\par Note that, by standard results \cite{clingher3, kondo, shapiro, shioda2} on elliptic fibrations on K3 surfaces, 
once a K3 surface ${\rm X}$ is endowed with an 
elliptic fibration $\varphi_{{\rm X}} \colon {\rm X} \rightarrow \mathbb{P}^1$ with two disjoint sections $S_1$ and $S_2$, the condition for the 
triple $(\varphi_{{\rm X}}, S_1, S_2)$ to define a Van Geemen-Sarti involution can be formulated entirely in terms of cohomology. One first considers 
the cohomology class $F$ of the fiber of $\varphi_{{\rm X}}$ as well as the class of $S_1$. These classes span a primitive lattice embedding 
${\rm H} \hookrightarrow {\rm NS}({\rm X})$. In fact, the Neron-Severi lattice factors into an orthogonal direct product 
\begin{equation}
\label{nssplit}
{\rm NS}({\rm X}) \ = \ {\rm H} \oplus \mathcal{W} 
\end{equation} 
where $\mathcal{W}$ is a negative definite lattice of rank ${\rm p}_{{\rm X}}-2$. 
Denote by $\mathcal{W}_{{\rm root}}$ the sub-lattice spanned by the roots of $\mathcal{W}$. This sub-lattice is actually spanned by the irreducible 
components of the singular fibers of $\varphi_{{\rm X}}$ not meeting $S_1$. As proved by Shioda \cite{shioda2}, one has then an isomorphism of abelian 
groups:
\begin{equation}
{\rm MW}(\varphi, S_1) \ \simeq \ \mathcal{W} / \mathcal{W}_{{\rm root}}. 
\end{equation}   
Let $S_2^w \in \mathcal{W} $ be the image of the class $S_2$ under the projection $ {\rm NS}({\rm X}) \rightarrow \mathcal{W} $
associated with the factorization $(\ref{nssplit})$. Note that $S_2^w = S_2-S_1-2F$ and $S_2^w$ has self-intersection $-4$. One obtains the following criterion:
\begin{prop}
\label{critvgs}
The triple $(\varphi_{{\rm X}}, S_1, S_2)$ defines a Van Geemen-Sarti involution $\Phi \colon {\rm X} \rightarrow {\rm X} $ if and only if $2  S_2^w \in \mathcal{W}_{{\rm root}}$. 
\end{prop}     
\noindent We also note that a Van Geemen-Sarti involution on a K3 surface ${\rm X}$ is equivalent to a pseudo-ample polarization by 
the rank-ten lattice ${\rm H} \oplus {\rm N}$ where ${\rm N}$ is the rank-eight Nikulin lattice as defined by \cite{morrison}. The 
Nikulin construction defines a natural involution on the ten-dimensional moduli space of ${\rm H} \oplus {\rm N}$-polarized K3 surfaces.

\section{Outline of the Paper}
In this work we construct Van Geemen-Sarti involutions on two specific classes of algebraic 
K3 surfaces. The first class consists of 
algebraic K3 surfaces ${\rm X}$ endowed with a pseudo-ample lattice polarization: 
$$ i \colon {\rm H} \oplus {\rm E}_7 \oplus {\rm E}_7 \ \hookrightarrow \ {\rm NS}({\rm X}).$$
This polarization structure is equivalent geometrically to a jacobian elliptic fibration on ${\rm X}$ that has two singular 
fibers of Kodaira type ${\rm III}^*$ or higher. For details regarding the concept of lattice polarization, we refer the reader 
to Dolgachev's paper \cite{dolgachev1} or the previous work \cite{clingher3} of the authors. For the purposes of this paper, an 
additional genericity condition is introduced (Definition $\ref{genk36i}$ of Section $\ref{sectionk3}$).
\par The second class of K3 surfaces consists of a special collection of double sextic surfaces - we consider 
surfaces ${\rm Z}$ obtained as minimal 
resolutions of double covers of the projective plane $\mathbb{P}^2$ branched over a configuration $\mathcal{L}$ of 
six distinct lines. The lines are assumed to be so located that no three of them pass through the same common point. 
We also introduce an explicit condition for genericity of $\mathcal{L}$, as given by Definition $\ref{genericity1}$ of Section $\ref{pseudokummer}$.    
\par The main results of this paper are as follows:
\begin{theorem}
\label{mainn11}
The K3 surfaces ${\rm Z}$ and ${\rm X}$ introduced above carry canonically-defined Van Geemen-Sarti involutions, 
denoted $ \Phi_{{\rm Z}}$ or $\Phi_{{\rm X}}$, respectively. 
\end{theorem}
\begin{theorem}
\label{mainn22}
If genericity is assumed on both sides, then one has a bijective correspondence:
\begin{equation}
\label{correspk3}
 ({\rm Z}, \mathcal{L}) \ \longleftrightarrow \ ({\rm X}, i) 
\end{equation}
between the two classes of surfaces, with the two K3 surfaces involved being related by a pair of dual geometric two-isogenies 
\begin{equation}
\label{chartt123}
\xymatrix 
{
{\rm Z} \ar @(dl,ul) _{\Phi_{{\rm Z}}} 
\ar @/_0.5pc/ @{-->} _{{\rm p}_{\Phi_{{\rm Z}}}} [rr]
&
& {\rm X} \ar @(dr,ur) ^{\Phi_{{\rm X}}} 
\ar @/_0.5pc/ @{-->} _{{\rm p}_{\Phi_{{\rm X}}}} [ll] \\
}
\end{equation}
as described in Section $\ref{isogenydef}$. 
\end{theorem}
\noindent Theorem $\ref{mainn22}$ remains true if the genericity conditions are removed. However, in that case, in order to account for all 
possible $ {\rm H} \oplus {\rm E}_7 \oplus {\rm E}_7$-polarized K3 surfaces $({\rm X},i)$, one has to allow for surfaces ${\rm Z}$ to degenerate 
to situations when at least three of the six lines in the configuration $\mathcal{L}$ are meeting at a point. The proofs associated with these 
degenerate cases will be included in a subsequent paper.   
\par The present work builds on ideas from paper \cite{clingher4}, where the authors have shown that dual pairs of geometric two-isogenies as 
in Section $\ref{isogenydef}$ relate K3 surfaces ${\rm X}$ polarized by the rank-eighteen lattice ${\rm H} \oplus {\rm E}_8 \oplus {\rm E}_8$ 
to Kummer surfaces ${\rm Z}$ associated to a cartesian product of two elliptic curves. In this situation, the Van Geemen-Sarti involution 
$\Phi_{{\rm X}}$ is a Shioda-Inose structure. This case was also considered by T. Shioda in \cite{shioda}. In an earlier work motivated by arithmetic 
considerations, B. Van Geemen and J. Top \cite{geementop} have presented a particular variant of the ${\rm H} \oplus {\rm E}_8 \oplus {\rm E}_8$ case - 
an isogeny between a one-dimensional family of K3 surfaces polarized by ${\rm H} \oplus {\rm E}_8 \oplus {\rm E}_8 \oplus {\rm A}_1(2)$ and Kummer 
surfaces associated to a cartesian product a pair of of 
two-isogeneous elliptic curves. 
\par In the appendix section to paper \cite{galluzzi} by F. Galluzzi and G. Lombardo, I. Dolgachev argued that any K3 surface ${\rm Z}$ with the 
Neron-Severi lattice ${\rm NS}({\rm X})$ isomorphic to ${\rm H} \oplus {\rm E}_8 \oplus {\rm E}_7 $ carries a canonical Shioda-Inose structure 
and the associated Nikulin construction leads to a Kummer surface associated with the Jacobian ${\rm Jac}({\rm C})$ of a genus-two curve. This 
situation appears here as a particular case of Theorem $\ref{mainn22}$. Polarized K3 surfaces $({\rm X},i)$ 
for which the lattice polarization extends to ${\rm H} \oplus {\rm E}_8 \oplus {\rm E}_7 $ correspond, under $(\ref{correspk3})$, to 
configurations $\mathcal{L}$ in which the six lines are tangent to a common conic. An explicit formula for determining the 
${\rm H} \oplus {\rm E}_8 \oplus {\rm E}_7 $-polarized K3 surface ${\rm X}$ has been given by A. Kumar \cite{kumar}.         
\par We shall also note that the geometric setting of Theorems $\ref{mainn11}$ and $\ref{mainn22}$ is ideal for performing explicit 
Kuga-Satake type constructions \cite{kugasatake} without relying on period computations. In the companion paper \cite{clingher5}, the 
authors use the results of this work in order to give a full classification of the K3 surfaces polarized by the lattice 
${\rm H} \oplus {\rm E}_8 \oplus {\rm E}_7 $ in terms of Siegel modular forms.  
\par The first author would like to thank P. Rao and G.V. Ravindra for useful discussions related to various aspects of this work.  
The second author would like to thank J. de Jong, who encouraged the authors in their pursuit of the geometric structure behind 
the modular invariants in \cite{clingher4} and \cite{clingher5}.

\section{Double Covers of the Projective Plane}
\label{pseudokummer}
Let $\mathcal{L}= \{ {\rm L}_1, {\rm L}_2, \cdots {\rm L}_6 \}$ be a configuration of six distinct lines in $\mathbb{P}^2$. We shall assume 
that no three of the six lines are concurrent. Denote by $q_{ij}$, with $1\leq i < j \leq 6$, the fifteen resulting intersection points. 
Let $\rho \colon {\rm R} \rightarrow \mathbb{P}^2 $ be the blow-up of the projective plane at the 
points $q_{ij}$ and denote by ${\rm L}'_1, {\rm L}'_2, \cdots {\rm L}'_6$ 
the rational curves in ${\rm R}$ obtained as the proper transforms of the six lines ${\rm L}_1, {\rm L}_1, \cdots {\rm L}_6$. Since
$$ \frac{1}{2} \ \sum_{i=1}^6 \ {\rm L}'_i \ \in {\rm NS}({\rm R}), $$
one has that there exists a double cover $ \pi \colon {\rm Z} \rightarrow {\rm R}$ branched over ${\rm L}'_1, {\rm L}'_2, \cdots {\rm L}'_6$. The surface ${\rm Z}$ 
is a smooth algebraic K3 surface of Picard rank sixteen or higher. In this section, we prove that the K3 surface ${\rm Z}$ so defined carries a canonical Van Geemen-Sarti involution denoted $\Phi_{{\rm Z}}$. Moreover, the K3 surface ${\rm W}$ resulting from the Nikulin construction associated to the involution $\Phi_{{\rm Z}}$ is endowed with a canonical  
${\rm H} \oplus {\rm E}_7 \oplus {\rm E}_7$ polarization. 
\par In order to define the involution $\Phi_{{\rm Z}}$, we follow the guidelines of Section $\ref{isogenydef}$. We introduce first an underlying 
elliptic fibration $\varphi_{{\rm Z}} \colon {\rm Z} \rightarrow \mathbb{P}^1 $ with two sections. We then show that 
fiber-wise translation by the second section determines a Van Geemen-Sarti involution. 
\subsection{A Special Elliptic Fibration on ${\rm Z}$} 
\noindent By construction, the surface ${\rm Z}$ comes endowed with a non-symplectic involution 
$\sigma \colon  {\rm Z} \rightarrow {\rm Z}$. The fixed locus of $\sigma$ is given by six rational curves $\Delta_1, \Delta_2, \cdots \Delta_6$, 
representing the ramification locus of the double cover map $ \pi \colon {\rm Z} \rightarrow {\rm R}$. 
We denote by ${\rm E}_{ij}$ the fifteen exceptional curves on the surface ${\rm R}$ and by 
${\rm G}_{ij}$ their respective strict transforms on ${\rm Z}$. Set also $ {\rm T} = (\pi \circ \rho)^*{\rm H} $ where ${\rm H}$ is a 
hyperplane divisor on $\mathbb{P}^2$.
\par The following divisor on ${\rm R}$ will prove to be instrumental: 
\begin{equation}
\label{ruling1}
{\rm D} \ =  \ 5 \rho ^* ({\rm H}) - 3{\rm E}_{13}- 2 \left ( {\rm E}_{14} + {\rm E}_{25}+ {\rm E}_{26} \right ) - 
\left ( {\rm E}_{24}+ {\rm E}_{35}+ {\rm E}_{36}+{\rm E}_{56} \right )  .
\end{equation}
The linear system $\vert {\rm D} \vert $ corresponds to curves of degree five in $\mathbb{P}^2$ passing through the four points 
$q_{24}, q_{35}, q_{36}, q_{56} $, having double points at $q_{14}, q_{25}, q_{26}$ and a triple point at $q_{13}$. 
\begin{prop}
\label{prpencil}
One has $h^1({\rm R}, {\rm D}) = 2 $. The pencil $\vert {\rm D} \vert $ is base-point free and its generic member is a smooth rational curve. 
The induced morphism 
\begin{equation}
\label{rulingonrr}
\varphi_{\vert {\rm D} \vert} \colon {\rm R} \rightarrow \mathbb{P}^1 
\end{equation} 
is a ruling.  
\end{prop}
\begin{proof}
Note that it suffices to prove the above statement assuming that ${\rm R}$ is the blow-up of $\mathbb{P}^2$ at the eight points 
$q_{13}$, $q_{14}$, $q_{24}$, $q_{25}$, $q_{26}$, $q_{35}$, $q_{36}$, $q_{56} $. The eight points in question are in 
{\it almost general position} (as defined in \cite{demazure3}). Then, as proved in \cite{demazure3, coray}, the rational surface ${\rm R}$ 
is a generalized Del Pezzo surface with the anticanonical line bundle $ - {\rm K}_{{\rm R}}$ having the big and nef properties. 
\par Since ${\rm D}^2=0$ and ${\rm D} \cdot {\rm K}_{{\rm R}} = -2$, one obtains, via the Riemann-Roch formula:
$$ h^0({\rm R}, {\rm D})- h^1({\rm R}, {\rm D}) + h^2({\rm R}, {\rm D}) \ = \ 2.$$
But $ h^2({\rm R}, {\rm D}) = h^0({\rm R}, {\rm K}_{{\rm R}} - {\rm D}) = 0 $. In particular $h^0({\rm R}, {\rm D}) \geq 2$.
\par Let ${\rm C}$ be the unique conic in 
$\mathbb{P}^2$ passing through the five points $q_{13},q_{14}, q_{25}, q_{26}, q_{56}$. The conic ${\rm C}$ is smooth. 
Denote by ${\rm C}'$ the rational curve on ${\rm R}$ obtained as the proper transform of ${\rm C}$. Then: 
\begin{equation}
\label{squint3}
{\rm L}'_1 + {\rm L}'_2 + {\rm L}'_3 + {\rm C}'
\end{equation} 
is a special member of $ \vert {\rm D} \vert $. As 
${\rm D} \cdot {\rm L}'_1 = {\rm D} \cdot {\rm L}'_2 ={\rm D} \cdot {\rm L}'_3 = {\rm D} \cdot {\rm C}' = 0$, 
if $ \vert {\rm D} \vert $ were to have base points, then the entire divisor $(\ref{squint3})$ would be part of the base locus. 
This would imply $h^0({\rm R}, {\rm D}) = 1$, contradicting the above estimation. 
The pencil $\vert {\rm D} \vert$ has therefore no base points. By Bertini's Theorem, the generic member of $\vert {\rm D} \vert$ is smooth and irreducible, 
and by the degree-genus formula we obtain that the generic member is a smooth rational curve.  
\par It remains to be shown that $h^1({\rm R}, {\rm D})=0$. One has 
$ h^1({\rm R}, {\rm D}) = h^1({\rm R}, {\rm K}_{{\rm R}} - {\rm D}) $. But $ ( {\rm D} - {\rm K}_{{\rm R}} )^2=5 $ and since both 
${\rm D}$ and $  - {\rm K}_{{\rm R}} $ are nef, one has that ${\rm D} - {\rm K}_{{\rm R}}$ is nef. 
By Ramanujam's Vanishing Theorem \cite{ramanujam}, one obtains $h^1({\rm R}, {\rm K}_{{\rm R}} - {\rm D}) = 0$. 
\end{proof}
\noindent Note that the lines ${\rm L}'_5$ and ${\rm L}'_6$ are disjoint sections of the ruling $(\ref{rulingonrr})$, while ${\rm L}'_4$ is a bi-section.     
The entire construction lifts then to the level of the K3 surface ${\rm Z}$ where one obtains:   
\begin{lem}
The pull-back under the double cover $\rho \colon  {\rm Z} \rightarrow {\rm R}$ of the linear system associated to $(\ref{ruling1})$, i.e.  
$$ \vert \ 5{\rm T} - 3{\rm G}_{13}- 2 \left ( {\rm G}_{14}+{\rm G}_{25}+{\rm G}_{26} \right ) - 
\left ( {\rm G}_{24}+ {\rm G}_{35}+ {\rm G}_{36}+ {\rm G}_{56} \right ) \ \vert  , $$ 
determines an elliptic fibration 
$\varphi_{{\rm Z}} \colon {\rm Z} \rightarrow \mathbb{P}^1 $ with 
the smooth rational curves $\Delta_5$ and $\Delta_6$ as distinct sections. 
\end{lem}
\noindent The smooth fibers of $\varphi_{{\rm Z}}$ appear as double covers of the smooth rational curves of the ruling $(\ref{rulingonrr})$, with the branch 
locus given by the four points of intersection with ${\rm L}'_4$, ${\rm L}'_5$ and ${\rm L}'_6$. 
\par Let us discuss the basic properties of the elliptic fibration $\varphi_{{\rm Z}}$. We shall 
differentiate between the following two possibilities: 
\begin{itemize}
\item [(a)] the six lines of the configuration $\mathcal{L}$ are tangent to a common smooth conic in $\mathbb{P}^2$,  
\item [(b)] there is no smooth conic tangent to all the six lines of the configuration $\mathcal{L}$. 
\end{itemize} 
In situation (a) the surface ${\rm Z}$ is a Kummer surface associated to 
the Jacobian of a genus-two curve. We shall refer to such a six-line configuration as ${\bf special}$ 
or ${\bf Kummer}$. If the six-line configuration $\mathcal{L}$ is in situation (b), we shall refer to it as {\bf non-special} 
or {\bf non-Kummer}. 
\begin{prop}
\label{thebigfiber}
If the six-line configuration $\mathcal{L}$ is non-Kummer then the elliptic fibration 
$\varphi_{{\rm Z}} \colon {\rm Z} \rightarrow \mathbb{P}^1 $ has a singular fiber of type ${\rm I}_4^*$. This special fiber becomes of type 
${\rm I}_5^*$ in the Kummer case.  
\end{prop}
\begin{proof}
Let ${\rm C}$ and ${\rm C}'$ be the curves defined within the proof of Proposition $ \ref{prpencil}$. Note that, as a consequence of the classical 
theorems of Pascal and Brianchon, the six-line configuration 
$\mathcal{L}$ is Kummer if and only if the conic curve ${\rm C}$ passes through $q_{34}$. 
\par If the 
configuration $\mathcal{L}$ is non-Kummer, then under the double cover map 
$ \pi \colon {\rm Z} \rightarrow {\rm R}$, one has $\pi^* {\rm C}' =  \Gamma $ where $\Gamma $ is smooth rational curve. 
The involution $\sigma$ maps the curve $ \Gamma $ to itself, with two fixed points located at the points of intersection with 
$ \Delta_{3}$ and $ \Delta_{4}$, respectively. 
\par However, if the six-line configuration $\mathcal{L}$ is Kummer, then one has:
$$ \pi^* {\rm C}' =  \Gamma_1 + \Gamma_2 $$  
where $ \Gamma_1$, $ \Gamma_2 $ are two disjoint smooth rational curves. The two curves $ \Gamma_1$, $ \Gamma_2 $ are mapped 
one onto the other by the involution $\sigma$.
\par One obtains in this way a special configuration of rational curves on the K3 surface ${\rm Z}$. If $\mathcal{L}$ is not Kummer, 
we have the following dual diagram: 
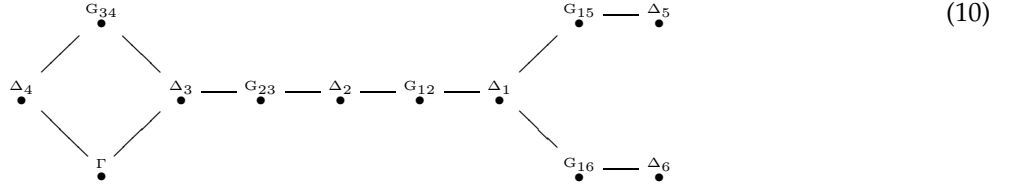
\begin{equation}
\label{diagg11}
\def\objectstyle{\scriptstyle}
\def\labelstyle{\scriptstyle}
\xymatrix @-0.9pc
{
 & \stackrel{{\rm G}_{34}}{\bullet} \ar @{-} [dr] & & & & & & \stackrel{{\rm G}_{15}}{\bullet} \ar @{-} [dl]
 & \stackrel{\Delta_{5}}{\bullet} \ar @{-} [l] \\
\stackrel{\Delta_{4}}{\bullet} \ar @{-} [dr] \ar @{-} [ur]& & \stackrel{\Delta_{3}}{\bullet} \ar @{-} [r] \ar @{-} [dl] &
\stackrel{{\rm G}_{23}}{\bullet} \ar @{-} [r] &
\stackrel{\Delta_{2}}{\bullet} \ar @{-} [r] &
\stackrel{{\rm G}_{12}}{\bullet} \ar @{-} [r] &
\stackrel{\Delta_{1}}{\bullet} \ar @{-} [dr] & \\
&  \stackrel{\Gamma}{\bullet} & & & &   & & \stackrel{{\rm G}_{16}}{\bullet} & \stackrel{\Delta_{6}}{\bullet} \ar @{-} [l]\\
}
\end{equation}
The special divisor:
$$ {\rm G}_{34} + \Gamma + 2 \left ( \Delta_3 + {\rm G}_{23} + \Delta _2+{\rm G}_{12} + \Delta_1 \right ) + {\rm G}_{15}+{\rm G}_{16}  $$
is the pull-back on ${\rm Z}$ of the special quintic curve $(\ref{squint3})$ and is a singular fiber of Kodaira type ${\rm I}_4^*$ for the elliptic 
fibration $\varphi_{{\rm Z}}$. The diagram above also includes the two sections $ \Delta_5$ and $\Delta_6$, 
as well as the bi-section $\Delta_4$.  
\par If the configuration $\mathcal{L}$ is Kummer, the dual diagram of rational curves gets modified as follows:
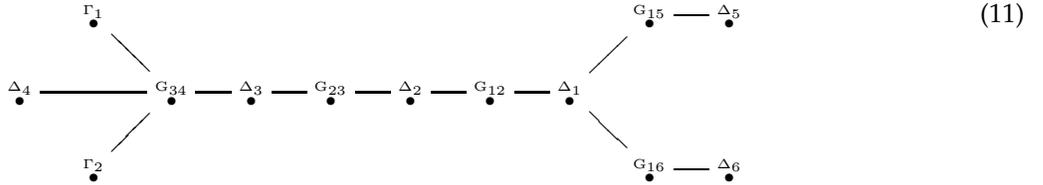
\begin{equation}
\label{diagg22}
\def\objectstyle{\scriptstyle}
\def\labelstyle{\scriptstyle}
\xymatrix @-0.9pc
{
& \stackrel{\Gamma_1}{\bullet} \ar @{-} [dr] & & & &  & & & \stackrel{{\rm G}_{15}}{\bullet} \ar @{-} [dl]
 & \stackrel{\Delta_{5}}{\bullet} \ar @{-} [l] \\
\stackrel{\Delta_{4}}{\bullet} \ar @{-} [rr]  & & \stackrel{{\rm G}_{34}}{\bullet} \ar @{-} [r] \ar @{-} [dl] &
\stackrel{\Delta_{3}}{\bullet} \ar @{-} [r] &
\stackrel{{\rm G}_{23}}{\bullet} \ar @{-} [r] &
\stackrel{\Delta_{2}}{\bullet} \ar @{-} [r] &
\stackrel{{\rm G}_{12}}{\bullet} \ar @{-} [r] &
\stackrel{\Delta_{1}}{\bullet} \ar @{-} [dr] & \\
& \stackrel{\Gamma_2}{\bullet} & & & &   & & & \stackrel{{\rm G}_{16}}{\bullet} & \stackrel{\Delta_{6}}{\bullet} \ar @{-} [l]\\
}
\end{equation}
The pull-back to ${\rm Z}$ of the quintic curve $(\ref{squint3})$ is now the divisor: 
$$ \Gamma_{1} + \Gamma_2 + 2 \left ( {\rm G}_{34}+ \Delta_3 + {\rm G}_{23} + \Delta _2+{\rm G}_{12} + \Delta_1 \right ) + {\rm G}_{15}+{\rm G}_{16}  $$
which forms a singular fiber of type ${\rm I}_5^*$ for the elliptic fibration $\varphi_{{\rm Z}}$.
\end{proof}
\noindent In addition to the special singular fiber of Proposition $\ref{thebigfiber}$, the elliptic fibration $\varphi_{{\rm Z}}$ carries additional 
singular fibers. In the generic situation, one has six additional ${\rm I}_2$ fibers plus two singular fibers of type  ${\rm I}_1$ in the non-Kummer case, 
or a single fiber of type ${\rm I}_1$ in the Kummer case, respectively. This genericity condition can be made precise. Consider the following divisors on the surface ${\rm R}$: 
\begin{align*}
\Phi_1 \ =& \ 5 \rho ^* ({\rm H}) - 3{\rm E}_{13}- 2 \left ( {\rm E}_{14} + {\rm E}_{25}+ {\rm E}_{26} \right ) - 
\left ( {\rm E}_{24}+ {\rm E}_{35}+ {\rm E}_{36}+{\rm E}_{45}+ {\rm E}_{56} \right ) \\
\Phi_2 \ =& \ 4 \rho ^* ({\rm H}) - 2 \left ( {\rm E}_{13}+ {\rm E}_{14} + {\rm E}_{25}  \right ) - 
\left ( {\rm E}_{24}+ {\rm E}_{26}+ {\rm E}_{35} + {\rm E}_{36}+ {\rm E}_{56} \right ) \\
\Phi_3 \ =& \ 3 \rho ^* ({\rm H}) - 2 {\rm E}_{13}   - 
\left ( {\rm E}_{14}+ {\rm E}_{24}+ {\rm E}_{25} + {\rm E}_{26}+ {\rm E}_{35} + {\rm E}_{56} \right ) \\
\Phi_4 \ =& \ 2 \rho ^* ({\rm H})  - 
\left ( {\rm E}_{13}  + {\rm E}_{14} + {\rm E}_{25} + {\rm E}_{26}+ {\rm E}_{35}  \right ) \\
\Phi_5 \ =& \  \rho ^* ({\rm H})  - 
\left ( {\rm E}_{13}   + {\rm E}_{25}   \right ) \\
\Phi_6 \ =& \ {\rm E}_{46}  \ .
\end{align*}
These classes have intersection numbers: $ \Phi_i^2 =  \Phi_i \cdot {\rm K}_{{\rm R}} = ({\rm D}-\Phi_i)^2  = ({\rm D}-\Phi_i) \cdot {\rm K}_{{\rm R}} = -1  $. 
In addition, 
$$h^0({\rm R},\Phi_i) = h^0({\rm R}, {\rm D} - \Phi_i ) =  1$$ 
for all indices $i$ with $ 1 \leq i \leq 6$.
\begin{dfn}
\label{genericity1}
The six-line configuration $\mathcal{L}$ is called generic if, for all $ 1 \leq i \leq 6$, the linear 
systems $\vert \Phi_i \vert$ and $ \vert {\rm D} - \Phi_i \vert $ each consist of a single smooth rational curve.
\end{dfn}
\noindent Assuming then a generic six-line configuration, one obtains that, for each $ 1 \leq i \leq 6$, the pull-back under 
the double-cover map $ \pi \colon {\rm Z} \rightarrow {\rm R}$ of the two rational curves associated with $\Phi_i$ and 
${\rm D}-\Phi_i$ provides a pair of rational curves on the K3 surface ${\rm Z}$ that form an ${\rm I}_2$ singular fiber for the 
elliptic fibration $\varphi_{{\rm Z}}$.   
\begin{rem}
Let us provide one example of non-generic situation. Consider the case when there exists an irreducible quintic curve in 
$\mathbb{P}^2$ with a triple point at $q_{13}$, three double points at $q_{14}$, $q_{25}$, $q_{26}$ and passing through $q_{24}$,  $q_{35}$, 
$q_{36}$, $q_{45}$, $q_{46}$, $q_{56}$. Then, the linear system
$$ \vert \ 5 \rho ^* ({\rm H}) - 3{\rm E}_{13}- 2 \left ( {\rm E}_{14} + {\rm E}_{25}+ {\rm E}_{26} \right ) - 
\left ( {\rm E}_{24}+ {\rm E}_{35}+ {\rm E}_{36}+{\rm E}_{45}+ {\rm E}_{46} + {\rm E}_{56} \right ) \ \vert $$ 
contains a single rational curve ${\rm M}$. The curve ${\rm M}$ does not meet $ {\rm L}'_i $ for any $ 1 \leq i \leq 6$ and 
$ \pi^*{\rm M} = \Xi_1 + \Xi_2 $  where $\Xi_1$, $ \Xi_2 $ are disjoint rational curves on the K3 surface ${\rm Z}$. The divisor:
$$ \Xi_1 + {\rm G}_{45} + \Xi_2 + {\rm G}_{46} $$
is then a fiber of type ${\rm I}_4$ in the elliptic fibration $\varphi_{{\rm Z}}$.
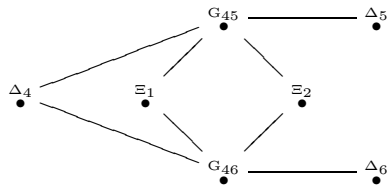
$$\def\objectstyle{\scriptstyle}
\def\labelstyle{\scriptstyle}
\xymatrix @-0.9pc
{
& & & \stackrel{{\rm G}_{45}}{\bullet} \ar @{-} [dr] & & \stackrel{\Delta_{5}}{\bullet} \ar @{-} [ll] \\
\stackrel{\Delta_{4}}{\bullet} \ar @{-} [urrr]  \ar @{-} [drrr] & & \stackrel{\Xi_{1}}{\bullet} \ar @{-} [ur] \ar @{-} [dr] &
 &
\stackrel{\Xi_{2}}{\bullet} \ar @{-} [dl]  \\
& & & \stackrel{{\rm G}_{46}}{\bullet} & & \stackrel{\Delta_{6}}{\bullet} \ar @{-} [ll]\\
}$$
\end{rem}
%
%
%
%
\begin{theorem}
\label{vgsth}
Let ${\rm Z}$ be a K3 surface associated with a generic six-line configuration $\mathcal{L}$. The section $\Delta_6$, interpreted 
as an element of the Mordell-Weil group ${\rm MW}(\varphi_{{\rm Z}}, \Delta_5)$, has order two. Fiber-wise translations by 
$\Delta_6$ in the smooth fibers of $\varphi_{{\rm Z}}$ extend to form a Van Geemen-Sarti 
involution $ \Phi_{{\rm Z}} \colon {\rm Z} \rightarrow {\rm Z}$.
\end{theorem}
\begin{proof}
In order to prove the above statement, one needs to verify the condition of Proposition $\ref{critvgs}$. We shall perform this 
verification here for the case of a non-Kummer 
line configuration $\mathcal{L}$. One can check that similar computation holds in the case of a Kummer configuration. 
\par Assume therefore that the six-line configuration $ \mathcal{L}$ is non-Kummer. Let ${\rm F}$ be the cohomology class of the fiber in $\varphi_{{\rm Z}}$, i.e.
$${\rm F} = 5{\rm T} - 3{\rm G}_{13}- 2 \left ( {\rm G}_{14}+{\rm G}_{25}+{\rm G}_{26} \right ) - 
\left ( {\rm G}_{24}+ {\rm G}_{35}+ {\rm G}_{36}+ {\rm G}_{56} \right ). $$ 
One has then the orthogonal direct product 
$$ {\rm NS}({\rm Z}) =  \langle {\rm F} , \Delta_5 \rangle  \oplus \mathcal{W}  .$$
The root sub-lattice $\mathcal{W}_{{\rm root}} \subset \mathcal{W} $ is spanned by the cohomology classes associated with the irreducible components of the singular fibers of 
$\varphi_{{\rm Z}}$ not meeting $\Delta_5$. The factorization of $\mathcal{W}_{{\rm root}}$ includes then the following:
$$ \langle \Upsilon_1 \rangle \oplus \langle  \Upsilon_2 \rangle  \oplus \langle \Upsilon_3 \rangle \oplus \ \cdots \ 
\oplus \langle \Upsilon_6 \rangle \oplus \langle  
\Upsilon_7, \ 
\Upsilon_8, \ \Delta_3, \ {\rm G}_{23}, \ \Delta_2, \ {\rm G}_{12}, \ \Delta_1, \ {\rm G}_{16} \rangle $$   
where $\Upsilon_i = \pi^* \Phi_i $, for $ 1 \leq i \leq 6$, and 
$$ \Upsilon_7 = {\rm G}_{34} \ , \ \ \ 
\Upsilon_8 = \Gamma = 2{\rm T} - 
\left ( {\rm G}_{13}+ {\rm G}_{14}+ {\rm G}_{25}+ {\rm G}_{26} + {\rm G}_{56}  \right ).$$
The six classes $\Upsilon_1, \ \Upsilon_2 \  \ \cdots \ \Upsilon_6$ represent the 
rational curves in the ${\rm I}_2$ singular fibers which do not meet $\Delta_5$. In this context, one has:
$$ \Delta_6^w \ = \ \Delta_6 - \Delta_5 - 2 {\rm F} \ = \ - \left ( \Delta_3 + {\rm G}_{23} + \Delta_2 + {\rm G}_{12} + \Delta_1 + {\rm G}_{16} \right )  -  \frac{1}{2} \left ( \Upsilon_1 + \Upsilon_2 + \cdots + \Upsilon_7 + \Upsilon_8 \right ) 
\ . $$  
Hence $ 2  \Delta_6^w \in \mathcal{W}_{{\rm root}}$.
\end{proof}
\noindent The above theorem remains true if one removes the genericity condition. Proofs for the non-generic cases will however 
not be included here. 
\begin{rem}
Note that, on each of the smooth fibers of the elliptic fibration $\varphi_{{\rm Z}}$, one has four distinct points given by the 
intersections with $\Delta_5$, $\Delta_6$ and $\Delta_4$. Consider the elliptic curve group law with center at $\Delta_5$. The intersection with 
$\Delta_6$ provides a special point of order two. The remaining two points of order two are located at the intersections with $\Delta_4$.
\end{rem}
%
%
%
%
%
%
\subsection{Properties of the Involution $\Phi_{{\rm Z}}$}
\noindent Let us discuss the Nikulin construction associated with the Van Geemen-Sarti involution $\Phi_{{\rm Z}}$. 
Note that, by construction, the involution $\Phi_{{\rm Z}}$ commutes with the non-symplectic involution $\sigma$. A second important feature is given 
by the fixed locus $\{ p_1, p_2, \cdots p_8 \} $ of $\Phi_{{\rm Z}}$. For simplicity of exposition, we shall assume that the six-line configuration is generic. \par Consider the case of a non-Kummer configuration $\mathcal{L}$. The rational curves $\Delta_4,$ $\Delta_3$, ${\rm G}_{23}$,  $\Delta_2$, ${\rm G}_{21}$, $\Delta_1$ get mapped to themselves under $\Phi_{{\rm Z}}$ and each of these six curves contains two of the fixed points. We denote by $p_2, p_3, p_4,p_5 $ the following four intersection points:
$$ \Delta_1 \cap {\rm G}_{21}, \ \Delta_2 \cap {\rm G}_{21}, \ \Delta_2 \cap {\rm G}_{23}, \  \Delta_3 \cap {\rm G}_{23}. $$
There are two additional fixed points $ p_1$, $p_6$ on $\Delta_1$, $\Delta_3$, respectively. The last two points $p_7$, $p_8$ are given by 
the singularities of the ${\rm I}_1$ fibers. Note that $p_7$, $p_8$ lie on $\Delta_4$.
\par If the six-line configuration $\mathcal{L}$ is Kummer, then the above set-up of 
the fixed locus $\{ p_1, p_2, \cdots p_8 \} $ gets modified slightly. 
One obtains $p_6$ as the intersection $\Delta_3 \cap {\rm G}_{34}$, the point $p_7$ lies on ${\rm G}_{34}$, and $p_8$ is the singularity of the single ${\rm I}_1$ fiber. It is still the case that $p_7$, $p_8$ lie on $\Delta_4$.  
\par Denote by ${\rm W}$ the K3 surface obtained from the Nikulin construction associated to the involution $\Phi_{{\rm Z}}$. By the general framework 
presented in Section $\ref{isogenydef}$, the surface ${\rm W}$ inherits a Jacobian elliptic fibration $\varphi_{{\rm W}}$. The singular fiber types 
of the fibration $\varphi_{{\rm W}}$ are: $ {\rm I}_{8}^* + 2 \times {\rm I}_2 + 6 \times {\rm I}_1 $ in the non-Kummer case and 
$ {\rm I}_{10}^* + {\rm I}_2 + 6 \times {\rm I}_1$ in the Kummer case, respectively. In order to be precise, consider the case of a non-Kummer configuration. 
In such a situation, the six rational curves 
\begin{equation}
\label{curves1}
 \Delta_1, \ \Delta_2, \ \Delta_3, \ \Delta_4,  \ {\rm G}_{12} , \ {\rm G}_{23} 
 \end{equation}
are mapped to themselves 
by the involution $\varphi_{{\rm Z}}$. The three pairs of disjoint curves 
\begin{equation}
\label{curves2}
(\Delta_5, \Delta_6), \ ({\rm G}_{15}, {\rm G}_{16}), \ ({\rm G}_{34}, \Gamma) 
\end{equation} 
are exchanged by $\varphi_{{\rm Z}}$. We denote by 
$$ \widetilde{\Delta}_1, \ \widetilde{\Delta}_2, \ \widetilde{\Delta}_3, \  \widetilde{\Delta}_4 , \  
\widetilde{{\rm G}}_{12}, \ \widetilde{{\rm G}}_{23}, \ \widetilde{\Delta}_{5}, \ \widetilde{{\rm G}}_{15}, \ \widetilde{\Gamma}$$ 
the nine rational curves on the surface ${\rm W}$ that arise as push-forward of the curves in $(\ref{curves1})$ and $(\ref{curves2})$. Let also
\begin{equation}
\label{eveneight1}
\Psi_1, \ \Psi_2, \ \Psi_3, \Psi_4, \ \Psi_5, \ \Psi_6, \ \Psi_7, \  \Psi_8 
\end{equation}
be the eight exceptional curves associated with the fixed locus. Two additional rational curves $\widetilde{{\rm J}}_7$, $ \widetilde{{\rm J}}_8$ appear 
from resolving the quotients of singular curves of the ${\rm I}_1$ fibers of $\varphi_{{\rm Z}}$ with singularities at $p_7$ and $p_8$. 
One obtains therefore nineteen rational curves on ${\rm W}$ that intersect according to the following dual diagram. 
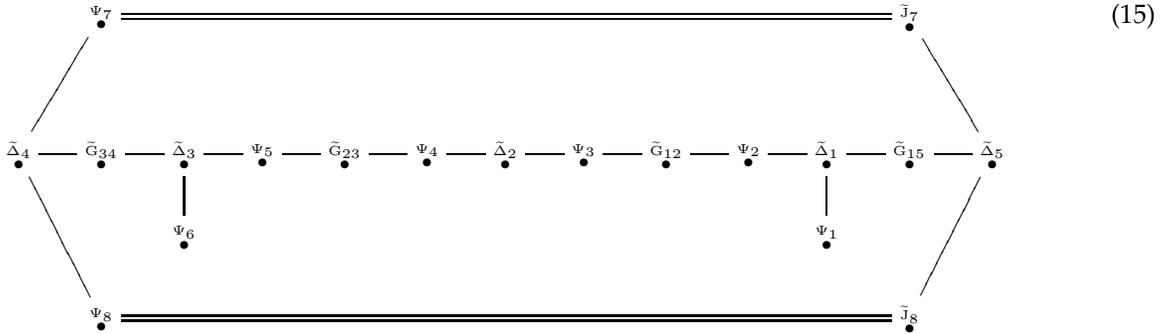
\begin{equation}
\label{diagg33}
\def\objectstyle{\scriptstyle}
\def\labelstyle{\scriptstyle}
\xymatrix @-0.8pc  {
& & \stackrel{\Psi_7}{\bullet} \ar @{=}[rrrrrrrrrr]& &  & & &    & & & & & \stackrel{\widetilde{{\rm J}}_7}{\bullet} &  &  \\ 
 & &  & &  & & & &   &    & &  & & \\
& \stackrel{\widetilde{\Delta}_4}{\bullet} \ar @{-} [r] \ar @{-}[ddr] \ar @{-}[uur] &
\stackrel{\widetilde{{\rm G}}_{34}}{\bullet} \ar @{-} [r]  &
\stackrel{\widetilde{\Delta}_3}{\bullet} \ar @{-} [r]  \ar @{-} [d] &
\stackrel{\Psi_5}{\bullet} \ar @{-} [r] &
\stackrel{\widetilde{{\rm G}}_{23}}{\bullet} \ar @{-} [r] &
\stackrel{\Psi_4}{\bullet} \ar @{-} [r] &
\stackrel{\widetilde{\Delta}_2}{\bullet} \ar @{-} [r] &
\stackrel{\Psi_3}{\bullet} \ar @{-} [r] &
\stackrel{\widetilde{{\rm G}}_{12}}{\bullet} \ar @{-} [r] &
\stackrel{\Psi_2}{\bullet} \ar @{-} [r] &
\stackrel{\widetilde{\Delta}_1}{\bullet} \ar @{-} [r] \ar @{-} [d] &
\stackrel{\widetilde{{\rm G}}_{15}}{\bullet} & \stackrel{\widetilde{\Delta}_5}{\bullet} \ar @{-} [l] \ar @{-}[ddl] \ar @{-}[uul]
 & \\
 & &  & \stackrel{\Psi_6}{\bullet} &  & & & &   &    & & \stackrel{\Psi_1}{\bullet}  & & \\
 & & \stackrel{\Psi_8}{\bullet} \ar @{=}[rrrrrrrrrr]& &  & & &    & & & & &   \stackrel{\widetilde{{\rm J}}_8}{\bullet} & & \\
} 
\end{equation}
The ${\rm I}_{8}^*$ singular fiber of $\varphi_{{\rm W}}$ is given by:
$$ \widetilde{{\rm G}}_{34} + \Psi_6 +  2 \left ( \widetilde{\Delta}_3 + \Psi_5 + \widetilde{{\rm G}}_{23} + \Psi_4 + 
\widetilde{\Delta}_2 + \Psi_3 + \widetilde{{\rm G}}_{12} + \Psi_2 + \widetilde{\Delta}_1 \right ) + \Psi_1 + \widetilde{{\rm G}}_{15}, $$
whereas $ \Psi_j + \widetilde{{\rm J}}_j $ with $j=7,8$ are fibers of type ${\rm I}_2$. The rational curves 
$\widetilde{\Delta}_4$, $ \widetilde{\Delta}_5$ are sections in $\varphi_{{\rm W}}$.    
\par If the six-line configuration is Kummer, then, using a notation along the same lines as before, the nineteen rational curves 
intersect in a slightly different manner.
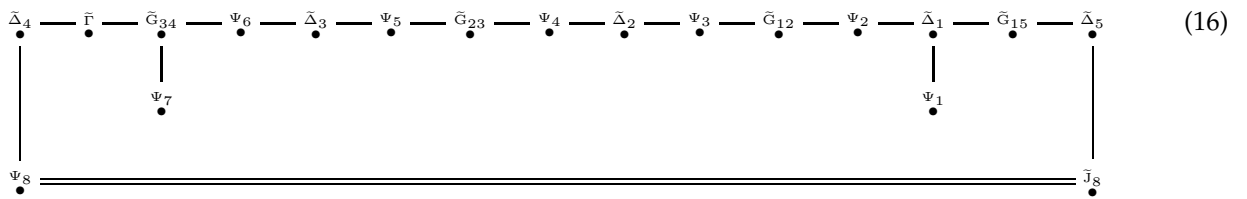
\begin{equation}
\label{diagg44}
\def\objectstyle{\scriptstyle}
\def\labelstyle{\scriptstyle}
\xymatrix @-0.9pc  {
& \stackrel{\widetilde{\Delta}_4}{\bullet} \ar @{-} [r] \ar @{-}[dd]  &
\stackrel{\widetilde{\Gamma}}{\bullet} \ar @{-} [r] &
\stackrel{\widetilde{{\rm G}}_{34}}{\bullet} \ar @{-} [r] \ar @{-} [d]&
\stackrel{\Psi_6}{\bullet} \ar @{-} [r]  &
\stackrel{\widetilde{\Delta}_3}{\bullet} \ar @{-} [r]   &
\stackrel{\Psi_5}{\bullet} \ar @{-} [r] &
\stackrel{\widetilde{{\rm G}}_{23}}{\bullet} \ar @{-} [r] &
\stackrel{\Psi_4}{\bullet} \ar @{-} [r] &
\stackrel{\widetilde{\Delta}_2}{\bullet} \ar @{-} [r] &
\stackrel{\Psi_3}{\bullet} \ar @{-} [r] &
\stackrel{\widetilde{{\rm G}}_{12}}{\bullet} \ar @{-} [r] &
\stackrel{\Psi_2}{\bullet} \ar @{-} [r] &
\stackrel{\widetilde{\Delta}_1}{\bullet} \ar @{-} [r] \ar @{-} [d] &
\stackrel{\widetilde{{\rm G}}_{15}}{\bullet} & \stackrel{\widetilde{\Delta}_5}{\bullet} \ar @{-} [l] \ar @{-}[dd] 
 & \\
 & & &  \stackrel{\Psi_7}{\bullet} &  & & & & &  &       &    & & \stackrel{\Psi_1}{\bullet}  & & \\
  & \stackrel{\Psi_8}{\bullet} \ar @{=}[rrrrrrrrrrrrrr]& &  & &  & & & & &     & & & & &   \stackrel{\widetilde{{\rm J}}_8}{\bullet}  & \\
} 
\end{equation}  
In diagram $(\ref{diagg44})$, $\widetilde{\Gamma}$ represents the push-forward of the rational curves $\Gamma_1$ and $\Gamma_2$. The 
${\rm I}_{10}^*$ singular fiber of $\varphi_{{\rm W}}$ is given by:
$$ \widetilde{\Gamma} + \Psi_7 +  2 \left ( \widetilde{{\rm G}}_{34} + \Psi_6 + \widetilde{\Delta}_3 + \Psi_5 + \widetilde{{\rm G}}_{23} + \Psi_4 + 
\widetilde{\Delta}_2 + \Psi_3 + \widetilde{{\rm G}}_{12} + \Psi_2 + \widetilde{\Delta}_1 \right ) + \Psi_1 + \widetilde{{\rm G}}_{15}, $$
with $ \Psi_8 + \widetilde{{\rm J}}_8 $ being a fiber of type ${\rm I}_2$. The rational curves 
$\widetilde{\Delta}_4$, $ \widetilde{\Delta}_5$ are still sections in $\varphi_{{\rm W}}$.  
%
%
%
\begin{theorem}
The K3 surface ${\rm W}$ associated to the involution $\Phi_{{\rm Z}}$ by the Nikulin construction carries a canonical 
pseudo-ample lattice polarization 
\begin{equation}
\label{observedpol}
i \colon {\rm H} \oplus {\rm E}_7 \oplus {\rm E}_7 \ \hookrightarrow \ {\rm NS}({\rm W}).
\end{equation}
If the six-line configuration ${\mathcal L}$ is Kummer, the lattice polarization $(\ref{observedpol})$ extends canonically to a polarization by the 
rank-seventeen lattice ${\rm H} \oplus {\rm E}_8 \oplus {\rm E}_7$.    
\end{theorem} 
\begin{proof}
We use the notation from diagrams $(\ref{diagg33})$ and $(\ref{diagg44})$. In 
the case of a Kummer six-line configuration, the primitive embedding of the orthogonal direct product 
${\rm H} \oplus {\rm E}_7 \oplus {\rm E}_7$ in ${\rm NS}({\rm W})$ is given by:
\begin{align*}
{\rm H} \ =& \ \langle \ \widetilde{\Delta}_2, \   
\Psi_7 + 2 \widetilde{\Delta}_4 + 3 \widetilde{{\rm G}}_{34} + 4\widetilde{\Delta}_3   + 2 \Psi_6 + 3 \Psi_5 + 2 \widetilde{{\rm G}}_{23}+ \Psi_4 \ 
 \rangle  \\ 
{\rm E}_7 \ =& \  \langle \ \Psi_7,   \widetilde{\Delta}_4 ,  \widetilde{{\rm G}}_{34} ,  \widetilde{\Delta}_3   ,  \Psi_6 ,  \Psi_5 ,  \widetilde{{\rm G}}_{23} \ 
 \rangle , \\ 
 {\rm E}_7 \ =& \  \langle \ \widetilde{{\rm J}}_8,   \widetilde{\Delta}_5 ,  \widetilde{{\rm G}}_{15} ,  \widetilde{\Delta}_1   ,  \Psi_1 ,  \Psi_2 ,  \widetilde{{\rm G}}_{12} \ 
 \rangle. 
\end{align*}
In the special case, one has a copy ${\rm H} \oplus {\rm E}_8 \oplus {\rm E}_7$ naturally embedded in ${\rm NS}({\rm W})$ as:
\begin{align*}
{\rm H} \ =& \ \langle \ \widetilde{\Delta}_2, \   
2 \widetilde{\Delta}_4 + 4 \widetilde{\Gamma} + 6 \widetilde{{\rm G}}_{34} + 3 \Psi_7 + 5 \Psi_6 + 4 \widetilde{\Delta}_3  + 3 \Psi_5 + 2 \widetilde{{\rm G}}_{23}+ \Psi_4 \ 
 \rangle  \\ 
{\rm E}_8  \ =& \   
\langle \ 
\widetilde{\Delta}_4 ,   \widetilde{\Gamma} ,  \widetilde{{\rm G}}_{34} , \Psi_7 ,  \Psi_6 ,  \widetilde{\Delta}_3  ,  \Psi_5 ,  \widetilde{{\rm G}}_{23} \ 
 \rangle , \\
{\rm E}_7 \ =& \ \langle \ \widetilde{{\rm J}}_8,   \widetilde{\Delta}_5 ,  \widetilde{{\rm G}}_{15} ,  \widetilde{\Delta}_1   ,  \Psi_1 ,  \Psi_2 ,  \widetilde{{\rm G}}_{12} \ 
 \rangle. 
 \end{align*}
\end{proof}
\noindent The results of this section show that 
every K3 surface ${\rm Z}$, obtained as the minimal resolution of a double cover of the projective plane 
$\mathbb{P}^2$ branched over 
a six-line configuration $\mathcal{L}$, is part of a geometric two-isogeny, in the sense of Section $\ref{isogenydef}$. 
The geometric counterpart of 
${\rm Z}$ under this isogeny is a K3 surface ${\rm W}$ carrying a canonical polarization by the rank-sixteen lattice 
${\rm H} \oplus {\rm E}_7 \oplus {\rm E}_7$. However, the results do not imply that all K3 surfaces endowed with 
${\rm H} \oplus {\rm E}_7 \oplus {\rm E}_7$-polarizations can be realized in this manner. This is clarified 
by the following section.

\section{K3 Surfaces Polarized by the Lattice ${\rm H} \oplus {\rm E}_7 \oplus {\rm E}_7$}
\label{sectionk3}
In this section ${\rm X}$ is an algebraic K3 surface endowed with a pseudo-ample lattice polarization 
\begin{equation}
\label{op}
i \colon {\rm H} \oplus {\rm E}_7 \oplus {\rm E}_7 \ \hookrightarrow \ {\rm NS}(X). 
\end{equation}
We shall also assume that the lattice polarization $(\ref{op})$ cannot be extended to a 
polarization by the rank-eighteen lattice $ {\rm H} \oplus {\rm E}_8 \oplus {\rm E}_8 $. It is known that a geometric 
two-isogeny as in Section $\ref{isogenydef}$ links any given K3 surface polarized by $ {\rm H} \oplus {\rm E}_8 \oplus {\rm E}_8 $ with 
the Kummer surface of a product of two elliptic curves and that the correspondence is bijective. This case was treated 
with full details in earlier works by the authors \cite{clingher4} as well as others \cite{inose,shioda}.       
\par In a manner similar to the presentation in the previous section, we shall distinguish between the following two possibilities: 
\begin{itemize}
\item [(a)] the lattice 
polarization $i$ can be extended to a polarization by the rank-seventeen lattice ${\rm H} \oplus {\rm E}_8 \oplus {\rm E}_7$ 
\item [(b)] the polarization $i$ cannot be extended to a polarization by the 
lattice $ {\rm H} \oplus {\rm E}_8 \oplus {\rm E}_7$. 
\end{itemize}
We shall refer to a polarized K3 surface $({\rm X},i)$ in situation (a) as 
${\bf special}$. A polarized K3 surface $({\rm X},i)$ satisfying condition (b) will be referred to as {\bf non-special}.
\subsection{Elliptic Fibrations on ${\rm X}$}
\noindent By standard results on elliptic fibrations on K3 surfaces (see discussion in \cite{clingher3} or related 
works \cite{kondo, shapiro}), jacobian elliptic fibrations on ${\rm X}$ are 
in one-to-one correspondence with isomorphism classes of primitive lattice embeddings of the rank-two hyperbolic lattice ${\rm H}$ into the 
Neron-Severi lattice ${\rm NS}({\rm X})$. There are at least four non-isomorphic primitive embeddings   
${\rm H} \hookrightarrow {\rm H} \oplus {\rm E}_7 \oplus {\rm E}_7 $, each of these embeddings leading via the polarization $i$ to a specific 
jacobian elliptic fibration on ${\rm X}$. Two of these 
embeddings/fibrations are particularly important for the discussion here. 
\begin{theorem}
\label{thhh123}
Let $({\rm X},i)$ be a K3 surface endowed with a pseudo-ample lattice polarization of type ${\rm H} \oplus {\rm E}_7 \oplus {\rm E}_7$. Then ${\rm X}$ 
carries two canonically defined jacobian elliptic fibrations 
$$\varphi^{{\rm s}}_{{\rm X}}, \varphi^{{\rm a}}_{{\rm X}} \colon {\rm X} \rightarrow \mathbb{P}^1, $$
which we shall refer as ${\bf standard}$ and ${\bf alternate}$. The standard fibration carries a section ${\rm S}^{\rm s}$. The alternate fibration 
carries two disjoint sections ${\rm S}^{{\rm a}}_1$ and ${\rm S}^{{\rm a}}_2$. 
\par If the polarized pair $({\rm X},i)$ is non-special, then the standard fibration has two singular fibers of type ${\rm III}^*$. In such 
a case the alternate fibration $\varphi^{{\rm a}}_{{\rm X}}$ has a singular fiber of 
type ${\rm I}_{8}^*$
\par If 
$({\rm X},i)$ is special, then the standard fibration has a singular fiber of type ${\rm II}^*$ and another fiber of type ${\rm III}^*$. 
The alternate fibration $\varphi^{{\rm a}}_{{\rm X}}$ carries a fiber of 
type ${\rm I}_{10}^*$ in this case.    
\end{theorem}
\begin{proof}
\par The first primitive lattice embedding of ${\rm H}$ is obvious - the first factor in the 
orthogonal decomposition of ${\rm H} \oplus {\rm E}_7 \oplus {\rm E}_7 $. This embedding induces then the 
canonical ${\bf stndard}$ elliptic fibration $\varphi_{{\rm X}}^{{\rm s}} \colon {\rm X} \rightarrow \mathbb{P}^1 $ with 
a section ${\rm S}^{{\rm s}}$ and two special fibers of Kodaira type ${\rm III}^*$ or higher. The pair 
$(\varphi_{{\rm X}}^{{\rm s}}, {\rm S}^{{\rm s}})$ is uniquely defined, up to an automorphism of ${\rm X}$. If $({\rm X},i)$ is non-special 
then $\varphi_{{\rm X}}^{{\rm s}}$ has two singular fibers of type ${\rm III}^*$ and one obtains a configuration of seventeen 
smooth rational curves as in the following dual diagram.
\begin{equation}
\label{diagg66}
\def\objectstyle{\scriptstyle}
\def\labelstyle{\scriptstyle}
\xymatrix @-0.9pc  {
\stackrel{a_1}{\bullet} \ar @{-} [r] 
& \stackrel{a_2}{\bullet} \ar @{-} [r]&
\stackrel{a_3}{\bullet} \ar @{-} [r]  &
\stackrel{a_4}{\bullet} \ar @{-} [r] \ar @{-} [d] &
\stackrel{a_6}{\bullet} \ar @{-} [r] &
\stackrel{a_7}{\bullet} \ar @{-} [r] &
\stackrel{a_8}{\bullet} \ar @{-} [r] &
\stackrel{S^{{\rm s}}}{\bullet} \ar @{-} [r] &
\stackrel{b_8}{\bullet} \ar @{-} [r] &
\stackrel{b_7}{\bullet} \ar @{-} [r] &
\stackrel{b_6}{\bullet} \ar @{-} [r] &
\stackrel{b_4}{\bullet} \ar @{-} [r] \ar @{-} [d] &
\stackrel{b_3}{\bullet} & \stackrel{b_2}{\bullet} \ar @{-} [l]
 & \stackrel{b_1}{\bullet} \ar @{-} [l] \\
 & &  & \stackrel{a_5}{\bullet} &  & & & &   &    & & \stackrel{b_5}{\bullet}  & &   \\
} 
\end{equation}
\noindent The standard embedding of ${\rm H}$ is spanned by $ \{ {\rm F}^{{\rm s}},{\rm S}^{{\rm s}} \}$ where 
\begin{equation}
\label{ffss}
{\rm F}^{{\rm s}}= a_1+2a_2+3a_3+4a_4+2a_5+3a_6+2a_7+a_8= b_1+2b_2+3b_3+4b_4+2b_5+3b_6+2b_7+b_8,
\end{equation}
and the two ${\rm E}_7$ sub-lattices are spanned by $\{ a_1, a_2, \cdots a_7 \}$ and $\{ b_1, b_2, \cdots b_7 \}$, respectively. 
\par In the case where $({\rm X},i)$ is special, the standard fibration $\varphi_{{\rm X}}^{{\rm s}} $ has two singular fibers of types ${\rm II}^*$  and ${\rm III}^*$, respectively. An extra rational curve appears 
on the dual diagram. 
\begin{equation}
\label{diagg55}
\def\objectstyle{\scriptstyle}
\def\labelstyle{\scriptstyle}
\xymatrix @-0.9pc  {
\stackrel{a_1}{\bullet} \ar @{-} [r] 
& \stackrel{a_2}{\bullet} \ar @{-} [r]&
\stackrel{a_3}{\bullet} \ar @{-} [r] \ar @{-} [d] &
\stackrel{a_5}{\bullet} \ar @{-} [r]  &
\stackrel{a_6}{\bullet} \ar @{-} [r] &
\stackrel{a_7}{\bullet} \ar @{-} [r] &
\stackrel{a_8}{\bullet} \ar @{-} [r] &
\stackrel{a_9}{\bullet} \ar @{-} [r] &
\stackrel{S^{{\rm s}}}{\bullet} \ar @{-} [r] &
\stackrel{b_8}{\bullet} \ar @{-} [r] &
\stackrel{b_7}{\bullet} \ar @{-} [r] &
\stackrel{b_6}{\bullet} \ar @{-} [r] &
\stackrel{b_4}{\bullet} \ar @{-} [r] \ar @{-} [d] &
\stackrel{b_3}{\bullet} & \stackrel{b_2}{\bullet} \ar @{-} [l]
 & \stackrel{b_1}{\bullet} \ar @{-} [l] \\
 &   & \stackrel{a_4}{\bullet} & & & & & & &  &    & & \stackrel{b_5}{\bullet}  & &   \\
} 
\end{equation}
\noindent In both diagrams $(\ref{diagg66})$ and $(\ref{diagg55})$ one sees a singular fiber of $D$-type. This fact leads 
one to a second primitive lattice embedding of ${\rm H}$ into $ {\rm NS}({\rm X})$.  The image of this 
embedding is spanned by $ \{ {\rm F}^{{\rm a}}, {\rm S}^{{\rm a}}_1 \}$ with these classes given, if $({\rm X},i)$ is non-special situation, by:
\begin{equation}
\label{ffaaa}
{\rm S}^{{\rm a}}_1= a_2, \ \ \ {\rm F}^{{\rm a}}=a_3 + a_5 + 2 (a_4+a_6+a_7+a_8+S^{{\rm s}}+b_8+b_7+b_6+ b_4) + b_3 + b_5 \ .   
\end{equation}  
In the special case one rather has:
$$ {\rm S}^{{\rm a}}_1= a_1, \ \ \ {\rm F}^{{\rm a}}=a_2 + a_4 + 2 (a_3+a_5+a_6+a_7+a_8+a_9+S^{{\rm s}}+b_8+b_7+b_6+ b_4) + b_3 + b_5 \ .   $$
This new embedding determines an alternate jacobian elliptic fibration $\varphi^{{\rm a}}_{{\rm X}} \colon {\rm X} \rightarrow \mathbb{P}^1 $. This 
fibration has two disjoint sections ${\rm S}^{{\rm a}}_1$ and ${\rm S}^{{\rm a}}_2$ obtained as ${\rm S}^{{\rm a}}_1 = a_2$, ${\rm S}^{{\rm a}}_2=b_2$, in 
the non-special case, and as ${\rm S}^{{\rm a}}_1 = a_1$, ${\rm S}^{{\rm a}}_2=b_2$ in the special case.
\end{proof}
\noindent The alternate fibration $\varphi^{{\rm a}}_{{\rm X}} \colon {\rm X} \rightarrow \mathbb{P}^1 $ plays a central role in the 
next results. In addition to the singular fiber of type ${\rm I}_8^*$ (or ${\rm I}_{10}^*$ if the polarization is special), the fibration 
$\varphi^{{\rm a}}_{{\rm X}}$ carries additional singular fibers. Generally, the singular fiber 
types of $\varphi^{{\rm a}}_{{\rm X}}$ are ${\rm I}^*_{8}+2 \times {\rm I}_2+6 \times {\rm I}_1$ in the case of a 
non-special polarized pair $({\rm X},i)$, and ${\rm I}^*_{10}+{\rm I}_2+ 6 \times {\rm I}_1$ for a  special $({\rm X},i)$, 
respectively. In such a general situation, the dual diagrams $(\ref{diagg66})$ and $(\ref{diagg55})$ get augmented 
(with two or one rational curves, respectively) to nineteen-curve diagrams as follows. 
\begin{equation}
\label{diagg77}
\def\objectstyle{\scriptstyle}
\def\labelstyle{\scriptstyle}
\xymatrix @-0.9pc  {
& & \stackrel{a_1}{\bullet} \ar @{=}[rrrrrrrrrr]& &  & & &    & & & & & \stackrel{d}{\bullet} &  &  \\ 
 & &  & &  & & & &   &    & &  & & \\
& \stackrel{a_2}{\bullet} \ar @{-} [r] \ar @{-}[ddr] \ar @{-}[uur] &
\stackrel{a_3}{\bullet} \ar @{-} [r]  &
\stackrel{a_4}{\bullet} \ar @{-} [r]  \ar @{-} [d] &
\stackrel{a_6}{\bullet} \ar @{-} [r] &
\stackrel{a_7}{\bullet} \ar @{-} [r] &
\stackrel{a_8}{\bullet} \ar @{-} [r] &
\stackrel{S}{\bullet} \ar @{-} [r] &
\stackrel{b_8}{\bullet} \ar @{-} [r] &
\stackrel{b_7}{\bullet} \ar @{-} [r] &
\stackrel{b_6}{\bullet} \ar @{-} [r] &
\stackrel{b_4}{\bullet} \ar @{-} [r] \ar @{-} [d] &
\stackrel{b_3}{\bullet} & \stackrel{b_2}{\bullet} \ar @{-} [l] \ar @{-}[ddl] \ar @{-}[uul]
 & \\
 & &  & \stackrel{a_5}{\bullet} &  & & & &   &    & & \stackrel{b_5}{\bullet}  & & \\
 & & \stackrel{c}{\bullet} \ar @{=}[rrrrrrrrrr]& &  & & &    & & & & &   \stackrel{b_1}{\bullet} & & \\
} 
\end{equation}
\begin{equation}
\label{diagg88}
\def\objectstyle{\scriptstyle}
\def\labelstyle{\scriptstyle}
\xymatrix @-0.9pc  {
& \stackrel{a_1}{\bullet} \ar @{-} [r] \ar @{-}[dd]  
& \stackrel{a_2}{\bullet} \ar @{-} [r]&
\stackrel{a_3}{\bullet} \ar @{-} [r] \ar @{-} [d] &
\stackrel{a_5}{\bullet} \ar @{-} [r]  &
\stackrel{a_6}{\bullet} \ar @{-} [r] &
\stackrel{a_7}{\bullet} \ar @{-} [r] &
\stackrel{a_8}{\bullet} \ar @{-} [r] &
\stackrel{a_9}{\bullet} \ar @{-} [r] &
\stackrel{S^{{\rm s}}}{\bullet} \ar @{-} [r] &
\stackrel{b_8}{\bullet} \ar @{-} [r] &
\stackrel{b_7}{\bullet} \ar @{-} [r] &
\stackrel{b_6}{\bullet} \ar @{-} [r] &
\stackrel{b_4}{\bullet} \ar @{-} [r] \ar @{-} [d] &
\stackrel{b_3}{\bullet} \ar @{-} [r] & \stackrel{b_2}{\bullet} \ar @{-}[dd]
 & \\
 & & &  \stackrel{a_4}{\bullet} &  & & & & &  &       &    & & \stackrel{b_5}{\bullet}  & & \\
  & \stackrel{c}{\bullet} \ar @{=}[rrrrrrrrrrrrrr]& &  & &  & & & & &     & & & & &   \stackrel{b_1}{\bullet}  & \\
} 
\end{equation}
Note the similarity with 
diagrams $(\ref{diagg33})$ and $(\ref{diagg44})$
\begin{theorem}
\label{vgse7th}
The section $S^{{\rm a}}_2$, interpreted as an element of the Mordell-Weil group ${\rm MW}(\varphi^{{\rm a}}_{{\rm X}}, S^{{\rm a}}_1)$, has order two. Fiber-wise translations by $S^{{\rm a}}_2$ extend to a Van Geemen-Sarti involution $ \Phi_{{\rm X}} \colon {\rm X} \rightarrow {\rm X}$.
\end{theorem}
\begin{proof}
One needs to verify the criterion of Proposition $\ref{critvgs}$. We shall do this check assuming a non-special polarization $({\rm X},i)$. 
Similar arguments hold for the special polarizations.
\par Assume that $({\rm X}, i)$ is a non-special polarization and take the orthogonal decomposition 
$${\rm NS}({\rm X}) = \langle {\rm F}^{{\rm a}}, a_2 \rangle \oplus \mathcal{W} .$$ 
This provides the negative-definite 
lattice $\mathcal{W}$ which has rank ${\rm p}_{{\rm X}}-2$. The root sublattice $\mathcal{W}_{{\rm root}}$ contains as orthogonal factors:
\begin{equation}
\langle a_4, \ a_5, \ a_6, \ a_7, \ a_8 , \ S^{{\rm s}}, \ b_8, \ b_7, \ b_6, \ b_5, \ b_4,  \ b_3 \rangle 
\ \oplus \ \langle b_1, \ \cdots  \rangle 
\ \oplus \langle d_1, \ d_2 \cdots  \rangle .  
\end{equation}  
The second factor above is spanned by the classes of the irreducible components of the singular fiber in $\varphi^{{\rm a}}_{{\rm X}}$ containing $b_1$ and 
not meeting $S^{{\rm a}}_1$. The third factor $ \langle d_1, \ d_2 \cdots  \rangle $ is spanned by the irreducible components of the 
singular fiber containing $a_1$ and not meeting $S^{{\rm a}}_1$. For a generic non-special $({\rm X},i)$, one has 
$\langle b_1, \ \cdots  \rangle = \langle b_1  \rangle $ and $\langle d_1, \ d_2 \cdots  \rangle = \langle d \rangle $ where $d$ is the rational curve of 
diagram $(\ref{diagg77})$.
\par One needs to check that $2b_2-2a_2-4{\rm F}^{{\rm a}} \in \mathcal{W}_{{\rm root}} $. By taking into account $(\ref{ffss})$ one obtains:
\begin{equation}
\label{ll1}
2b_2 - {\rm F}^{{\rm s}} \ = \ - \left ( b_8 + 2b_7+3b_6+4b_4+2b_5+3b_3 +b_1  \right ) \ \in \mathcal{W}_{{\rm root}}
\end{equation}
\begin{equation}
\label{ll2}
{\rm F}^{{\rm s}}- \left ( a_1 +2a_2+3a_3 \right )   \ = \  \left ( 4a_4+2a_5+3a_6 +2a_7 + a_8  \right ) \ \in \mathcal{W}_{{\rm root}}.
\end{equation}
Taking the sum of $(\ref{ll1})$ and $(\ref{ll1})$, we have:
\begin{equation}
\label{ll3}
\left ( 2b_2-2a_2 \right ) - \left ( a_1 +3a_3 \right )   \ \in \mathcal{W}_{{\rm root}}.
\end{equation}
Note also that, by comparing with $(\ref{ffaaa})$, we also have: 
\begin{equation}
\label{ll4}
{\rm F}^{{\rm a}}- a_3   \ = \  a_5 + 2 \left ( a_4+a_6+a_7+a_8+S^{{\rm s}}+b_8+b_7+b_6+ b_4 \right ) + b_3 + b_5 \ \in \mathcal{W}_{{\rm root}}
\end{equation}
\begin{equation}
\label{ll5}
{\rm F}^{{\rm a}}- a_1   \  \in \langle d_1, \ d_2 \cdots  \rangle \ \subset  \   \mathcal{W}_{{\rm root}}.
\end{equation}
One obtains therefore that:
\begin{equation}
\label{ll6}
4  {\rm F}^{{\rm a}} - \left ( a_1 +3a_3 \right )   \ \in \mathcal{W}_{{\rm root}}.
\end{equation}
By subtracting $(\ref{ll6})$ from $(\ref{ll3})$, we obtain $2b_2-2a_2-4{\rm F}^{{\rm a}} \in \mathcal{W}_{{\rm root}} $. 
\end{proof}
\subsection{Properties of the Involution $\Phi_{{\rm X}}$}
\label{k3invpropr}
We assume that the polarized K3 surface $({\rm X},i)$ is such that, in both cases (non-special or special), 
the alternate elliptic fibration $\varphi^{{\rm a}}_{{\rm X}}$ has singular fiber types ${\rm I}^*_8+2 \times {\rm I}_2+6 \times {\rm I}_1$ or 
${\rm I}^*_{10}+ {\rm I}_2+6 \times {\rm I}_1$, respectively. We shall therefore make use of the diagrams of rational curves 
$(\ref{diagg77})$ and $(\ref{diagg88})$ 
\par The fixed locus $  \{ n_1, n_2, n_3 , \cdots n_8 \}$ of the Van Geemen-Sarti involution $\Phi_{{\rm X}}$ appears as follows. 
The first six points $ n_1, n_2, n_3 , \cdots n_6 $ are 
the singularities of the six ${\rm I}_1$ fibers of the alternate fibration. The remaining $n_7$, $n_8$ are distinct points lying on 
the rational curve $S^{{\rm s}}$, if the polarization $({\rm X},i)$ is non-special, and on the curve $a_9$ if $({\rm X},i)$ is special. 
\par Two additional effective reduced divisors on ${\rm X}$ play a role in the construction. The first divisor, denoted $ {\rm Q}$ is obtained from compactifying the 
set of order-two points in the smooth fibers of $\varphi^{{\rm a}}_{{\rm X}}$ that do not lie on ${\rm S}^{{\rm a}}_1$ or ${\rm S}^{{\rm a}}_2$. 
The divisor ${\rm Q}$ is a bi-section of the alternate fibration. It contains $ n_1, n_2, n_3, n_4, n_5, n_6$ but not $n_1, n_2$. Generally, ${\rm Q}$ is 
a smooth genus-two curve and the restriction of the alternate fibration provides a double cover ${\rm Q} \rightarrow \mathbb{P}^1$ ramified at 
the six points $ n_1, n_2, n_3, n_4, n_5, n_6$. The second divisor, denoted ${\rm K}$ is obtained from compactifying the points $x$ 
in the smooth fibers of $\varphi^{{\rm a}}_{{\rm X}}$ that, with respect to the elliptic group law with neutral element at ${\rm S}^{{\rm a}}_1$, 
satisfy $2x= {\rm S}^{{\rm a}}_2$. The divisor ${\rm K}$ is a four-section of the alternate fibration. It contains all eight points of 
the fixed locus of $\Phi_{{\rm X}}$. Generally, ${\rm K}$ is a smooth curve of genus three in the non-special case and of genus two in 
the special case, respectively. The restriction of the alternate fibration gives a four-sheeted cover ${\rm K} \rightarrow \mathbb{P}^1$ 
branched at the base-points corresponding to singular fibers in the alternate fibration (nine points in the 
non-special case and eight points in the special case, respectively). Both divisors ${\rm Q}$ and ${\rm K}$ are mapped to themselves 
by the involution $\Phi_{{\rm X}}$. Their intersections with the curves of the big singular fiber of the alternate fibration are presented in the 
diagrams below. The first diagram corresponds to the case of a non-special $({\rm X},i)$. The second is associated with the special case.    
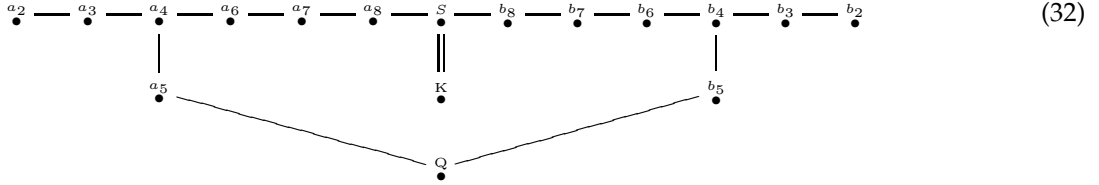
\begin{equation}
\label{diagg1010}
\def\objectstyle{\scriptstyle}
\def\labelstyle{\scriptstyle}
\xymatrix @-0.9pc  {
& \stackrel{a_2}{\bullet} \ar @{-} [r] 
&
\stackrel{a_3}{\bullet} \ar @{-} [r]  &
\stackrel{a_4}{\bullet} \ar @{-} [r]  \ar @{-} [d] &
\stackrel{a_6}{\bullet} \ar @{-} [r] &
\stackrel{a_7}{\bullet} \ar @{-} [r] &
\stackrel{a_8}{\bullet} \ar @{-} [r] &
\stackrel{S}{\bullet} \ar @{-} [r] \ar @{=}[d] &
\stackrel{b_8}{\bullet} \ar @{-} [r] &
\stackrel{b_7}{\bullet} \ar @{-} [r] &
\stackrel{b_6}{\bullet} \ar @{-} [r] &
\stackrel{b_4}{\bullet} \ar @{-} [r] \ar @{-} [d] &
\stackrel{b_3}{\bullet} & \stackrel{b_2}{\bullet} \ar @{-} [l] 
 & \\
 & &  & \stackrel{a_5}{\bullet} \ar @{-}[drrrr] &  & & & \stackrel{{\rm K}}{\bullet} &   &    & & \stackrel{b_5}{\bullet} \ar @{-}[dllll] & & \\
  & &  &  &  & &  & \stackrel{{\rm Q}}{\bullet} &   &    & &   & & \\
} 
\end{equation} 
\begin{equation}
\label{diagg1111}
\def\objectstyle{\scriptstyle}
\def\labelstyle{\scriptstyle}
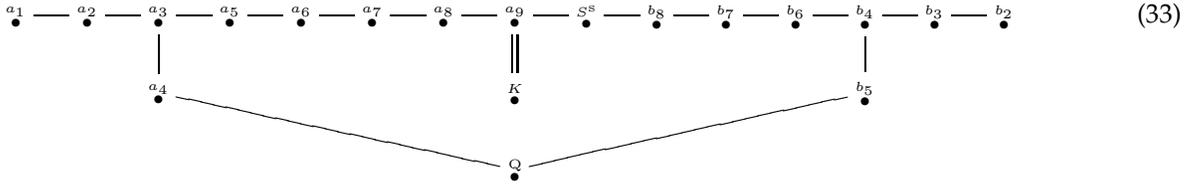
\xymatrix @-0.9pc  {
& \stackrel{a_1}{\bullet} \ar @{-} [r] 
& \stackrel{a_2}{\bullet} \ar @{-} [r]&
\stackrel{a_3}{\bullet} \ar @{-} [r] \ar @{-} [d] &
\stackrel{a_5}{\bullet} \ar @{-} [r]  &
\stackrel{a_6}{\bullet} \ar @{-} [r] &
\stackrel{a_7}{\bullet} \ar @{-} [r] &
\stackrel{a_8}{\bullet} \ar @{-} [r] &
\stackrel{a_9}{\bullet} \ar @{-} [r] \ar @{=}[d] &
\stackrel{S^{{\rm s}}}{\bullet} \ar @{-} [r] &
\stackrel{b_8}{\bullet} \ar @{-} [r] &
\stackrel{b_7}{\bullet} \ar @{-} [r] &
\stackrel{b_6}{\bullet} \ar @{-} [r] &
\stackrel{b_4}{\bullet} \ar @{-} [r] \ar @{-} [d] &
\stackrel{b_3}{\bullet} \ar @{-} [r] & \stackrel{b_2}{\bullet} 
 & \\
 & & &  \stackrel{a_4}{\bullet} \ar @{-}[drrrrr] &  & & & & \stackrel{K}{\bullet} &  &       &    & & \stackrel{b_5}{\bullet} \ar @{-}[dlllll] & & \\
  & & &   &  & & & & \stackrel{{\rm Q}}{\bullet} &  &       &    & &  & & \\
} 
\end{equation}
The intersections of ${\rm Q}$ and ${\rm K}$ with the ${\rm I}_2$ fiber curves are as follows. In the non-special case, one has. 
$$ {\rm Q} \cdot a_1 = {\rm Q} \cdot d = {\rm Q} \cdot c = {\rm Q} \cdot a_1 = 1, \ \ \ 
{\rm K} \cdot a_1 = {\rm K} \cdot d = {\rm K} \cdot c = {\rm K} \cdot a_1 = 2. $$
In the special case:  
$$ {\rm Q} \cdot c = {\rm Q} \cdot b_1 = 1, \ \ \ 
{\rm K} \cdot c = {\rm K} \cdot b_1 = 2. $$
\begin{dfn}
\label{genk36i}
A K3 surface $({\rm X},i)$ polarized by the lattice ${\rm H} \oplus {\rm E}_7 \oplus {\rm E}_7$ is called {\bf generic} if the 
following two conditions are satisfied:
\begin{itemize}
\item [(a)] The alternate fibration $\varphi_{{\rm X}}^{{\rm a}} \colon {\rm X} \rightarrow \mathbb{P}^1 $ has singular fiber types 
${\rm I}^*_8+2 \times {\rm I}_2+6 \times {\rm I}_1$ or ${\rm I}^*_{10}+ {\rm I}_2+6 \times {\rm I}_1$ depending on whether $({\rm X},i)$ 
is non-special, or special, respectively.
\item [(b)] The effective divisors ${\rm Q}$ and ${\rm K}$ introduced above are both irreducible. 
\end{itemize} 
\end{dfn}
\noindent We are now in position to prove the following result:
\begin{theorem}
Let $({\rm X},i)$ be a generic K3 surface polarized by the lattice ${\rm H} \oplus {\rm E}_7 \oplus {\rm E}_7$. Denote by ${\rm Y}$ the K3 surface 
obtained by the Nikulin construction associated to the Van Geemen-Sarti involution $\Phi_{{\rm X}}$. Then, the surface ${\rm Y}$ is isomorphic 
to the minimal resolution of a double cover of the projective plane $\mathbb{P}^2$ branched at a six-line configuration $\mathcal{L}$. 
No three of the six lines are concurrent. If the polarization $({\rm X},i)$ is non-special then the six-line configuration $\mathcal{L}$ is non-Kummer. For 
special polarizations $({\rm X},i)$, the configuration $\mathcal{L}$ is Kummer.
\end{theorem}
\begin{proof}
We present a detailed proof for the case when $({\rm X},i)$ is a non-special generic polarization. The same set of ideas together with a slight 
modification of the arguments provide the proof in the generic special case. 
\par This proof uses the notation of diagrams $(\ref{diagg77})$ and 
$(\ref{diagg1010})$. Note that the Van Geemen-Sarti involution $\Phi_{{\rm X}}$ 
maps the three curves ${\rm S}$, ${\rm Q}$ and ${\rm K}$ to 
themselves and interchanges the following nine pairs of rational curves:
\begin{equation}
\label{pairs113}
({\rm a}_8, {\rm b}_8) , \ ({\rm a}_7, {\rm b}_7) , \ ({\rm a}_6, {\rm b}_6) , \ ({\rm a}_5, {\rm b}_5) , \ ({\rm a}_4, {\rm b}_4) , \ 
({\rm a}_3, {\rm b}_3) , \ ({\rm a}_2, {\rm b}_2) , \ ({\rm a}_1, {\rm d}) , \ ({\rm c}, {\rm b}_1).
\end{equation} 
Under the push-forward by the rational degree-two map $ {\rm X} \dashrightarrow {\rm Y}$ of the Nikulin construction, the three curves 
${\rm S}$, ${\rm Q}$ and ${\rm K}$, as well as the curves in the first seven pairs of $(\ref{pairs113})$ determine ten smooth rational curves. 
We shall denote these curves by:
\begin{equation}
\widetilde{{\rm S}}, \ \widetilde{{\rm Q}}, \ \widetilde{{\rm K}}, \ \widetilde{{\rm a}}_8, \ \widetilde{{\rm a}}_7, \ \widetilde{{\rm a}}_6, \ 
\widetilde{{\rm a}}_5, \ \widetilde{{\rm a}}_4, \ \widetilde{{\rm a}}_3, \ \widetilde{{\rm a}}_2. \ 
\end{equation}   
The last two pairs in $(\ref{pairs113})$ determine rational curves with a single ordinary node. These form two ${\rm I}_1$ fibers in the elliptic fibration 
$ \varphi^{{\rm a}}_{{\rm Y}} \colon {\rm Y} \rightarrow \mathbb{P}^1 $ induced from the alternate fibration on ${\rm X}$.   
\par Denote by ${\rm U}_i$, $1 \leq i \leq 8$, the rational curves on ${\rm Y}$ appearing as exceptional curves 
associated to the fixed points $p_i$. Let us also consider ${\rm V}_i$, $1 \leq i \leq 6$, as the resolutions of the ${\rm I}_1$ fibers 
(of the alternate fibration) with singularities at $p_i$. One obtains twenty-four smooth rational curves on ${\rm Y}$ whose intersection pattern 
is summarized by the following dual diagram.    
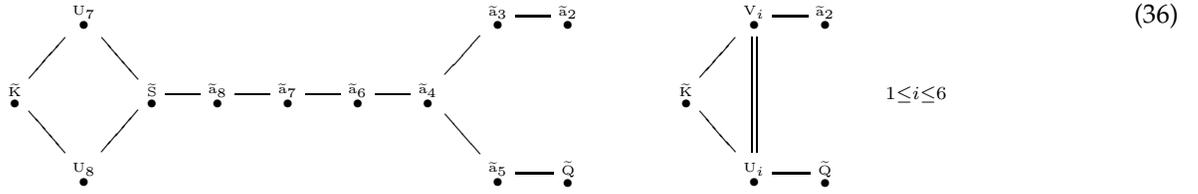
\begin{equation}
\label{diagg1212}
\def\objectstyle{\scriptstyle}
\def\labelstyle{\scriptstyle}
\xymatrix @-0.9pc
{
 & \stackrel{{\rm U}_{7}}{\bullet} \ar @{-} [dr] & & & & & & \stackrel{\widetilde{{\rm a}}_3}{\bullet} \ar @{-} [dl]
 & \stackrel{\widetilde{{\rm a}}_{2}}{\bullet} \ar @{-} [l] &  & &  \stackrel{{\rm V}_{i}}{\bullet} \ar @{=} [dd] \ar @{-} [r] & 
 \stackrel{\widetilde{{\rm a}}_{2}}{\bullet} & & & \\
\stackrel{\widetilde{{\rm K}}}{\bullet} \ar @{-} [dr] \ar @{-} [ur]& & \stackrel{\widetilde{{\rm S}}}{\bullet} \ar @{-} [r] \ar @{-} [dl] &
\stackrel{\widetilde{{\rm a}}_{8}}{\bullet} \ar @{-} [r] &
\stackrel{\widetilde{{\rm a}}_{7}}{\bullet} \ar @{-} [r] &
\stackrel{\widetilde{{\rm a}}_{6}}{\bullet} \ar @{-} [r] &
\stackrel{\widetilde{{\rm a}}_{4}}{\bullet} \ar @{-} [dr] & &  & & 
\stackrel{\widetilde{{\rm K}}}{\bullet} \ar @{-} [ur] \ar @{-} [dr]  & &   & 1 \leq i \leq  6 \\
&  \stackrel{{\rm U}_8}{\bullet} & & & &   & & \stackrel{\widetilde{{\rm a}}_{5}}{\bullet} & 
\stackrel{\widetilde{{\rm Q}}}{\bullet} \ar @{-} [l] & &  &  \stackrel{{\rm U}_{i}}{\bullet} \ar @{-} [r] & \stackrel{\widetilde{{\rm Q}}}{\bullet} &  & & \\
}
\end{equation}
The elliptic fibration $ \varphi^{{\rm a}}_{{\rm Y}} \colon {\rm Y} \rightarrow \mathbb{P}^1 $ has the singular fiber type 
$${\rm I}_4^*+6 \times {\rm I}_2 +2 \times {\rm I}_1 \ . $$ 
The two curves $\widetilde{{\rm a}}_2$, $\widetilde{{\rm Q}}$ form sections in $ \varphi^{{\rm a}}_{{\rm Y}}$ while $\widetilde{{\rm K}}$ is a bi-section. 
As explained in Section \ref{isogenydef}, fiber-wise translations by the section $\widetilde{{\rm Q}}$ determine the dual 
Van Geemen-Sarti involution $ \Phi_{{\rm Y}}$.
\par Let $\sigma \colon {\rm Y} \rightarrow {\rm Y}$ be the non-symplectic involution 
associated to $x \mapsto -x $ in the group law with origin at $\widetilde{{\rm a}}_2$ 
on the smooth fibers of $\varphi^{{\rm a}}_{{\rm Y}}$. By classical theory on elliptic curve deformations, the fixed locus of the 
involution $\sigma$ is given by the six disjoint rational curves:
\begin{equation}
\label{ttcurves1}
 \widetilde{{\rm K}}, \ \widetilde{{\rm S}}, \ \widetilde{{\rm a}}_{7}, \ \widetilde{{\rm a}}_{4}, \ \widetilde{{\rm a}}_{2}, \ \widetilde{{\rm Q}}. 
 \end{equation}
In addition, the rational curves
\begin{equation}
\label{ttcurves2}
\widetilde{{\rm a}}_{8} , \ \widetilde{{\rm a}}_{6}, \ \widetilde{{\rm a}}_{5}, \ \widetilde{{\rm a}}_{3}, \ 
{\rm U}_i \ {\rm with} \ 1 \leq i \leq 8, \ {\rm V}_i \ {\rm with} \ 1 \leq i \leq 6  
\end{equation}
are mapped onto themselves under $\sigma$. 
\par The quotient of the K3 surface ${\rm Y}$ by the involution $\sigma$ is a rational ruled surface ${\rm R}$ with a ruling 
\begin{equation}
\label{rulingonr}
\varphi_{{\rm R}} \colon {\rm R} \rightarrow \mathbb{P}^1 
\end{equation}
induced from the elliptic fibration $\varphi_{{\rm Y}}^{{\rm a}}$. We shall use the superscript $\ \hat{}\ $ to 
denote the rational curves on ${\rm R}$ obtained as push-forward under the quotient map of the curves in $(\ref{ttcurves1})$ and $(\ref{ttcurves2})$.    
A dual diagram similar to $(\ref{diagg1212})$ appears. 
\begin{equation}
\label{diagg1313}
\def\objectstyle{\scriptstyle}
\def\labelstyle{\scriptstyle}
\xymatrix @-0.9pc
{
 & \stackrel{\hat{{\rm U}}_{7}(-1)}{\bullet} \ar @{-} [dr] & & & & & & \stackrel{\hat{{\rm a}}_3(-1)}{\bullet} \ar @{-} [dl]
 & \stackrel{\hat{{\rm a}}_{2}(-4)}{\bullet} \ar @{-} [l] & \\
\stackrel{\hat{{\rm K}}(-4)}{\bullet} \ar @{-} [dr] \ar @{-} [ur]& & \stackrel{\hat{{\rm S}}(-4)}{\bullet} \ar @{-} [r] \ar @{-} [dl] &
\stackrel{\hat{{\rm a}}_{8}(-1)}{\bullet} \ar @{-} [r] &
\stackrel{\hat{{\rm a}}_{7}(-4)}{\bullet} \ar @{-} [r] &
\stackrel{\hat{{\rm a}}_{6}(-1)}{\bullet} \ar @{-} [r] &
\stackrel{\hat{{\rm a}}_{4}(-4)}{\bullet} \ar @{-} [dr] & & &  \\
&  \stackrel{\hat{{\rm U}}_8(-1)}{\bullet} & &    & & & & \stackrel{\hat{{\rm a}}_{5}(-1)}{\bullet} & 
\stackrel{\hat{{\rm Q}}(-4)}{\bullet} \ar @{-} [l] & \\
}
\end{equation}
$$
\def\objectstyle{\scriptstyle}
\def\labelstyle{\scriptstyle}
\xymatrix @-0.9pc
{
 &    \stackrel{\hat{{\rm V}}_{i}(-1)}{\bullet} \ar @{=} [dd] \ar @{-} [r] & 
 \stackrel{\hat{{\rm a}}_{2}(-4)}{\bullet} & & & \\ 
 \stackrel{\hat{{\rm K}}(-4)}{\bullet} \ar @{-} [ur] \ar @{-} [dr]  & &   & 1 \leq i \leq  6 \\
 &       \stackrel{\hat{{\rm U}}_{i}(-1)}{\bullet} \ar @{-} [r] & \stackrel{\hat{{\rm Q}}(-4)}{\bullet} &  & & \\
 }
$$
The self-intersection numbers are included.
\par The ruling $(\ref{rulingonr})$, as well as the rational curves of $(\ref{diagg1313})$, will be used to prove that the 
rational surface ${\rm R}$ is isomorphic to the blow-up of $\mathbb{P}^2$ at a configuration of fifteen distinct points corresponding to the 
intersection of six distinct lines. The considerations of Section $\ref{pseudokummer}$ bring some insight into the construction, allowing 
us to write explicitly the cohomology classes associated to the fifteen would-be exceptional curves.  
\begin{align*}
{\rm E}_{12} \ =& \  \hat{{\rm a}}_{2} \\ 
{\rm E}_{13} \  =& \  12 {\rm F}^{{\rm a}} + 3 \hat{{\rm a}}_{2} - 3\hat{{\rm a}}_{4} -3\hat{{\rm a}}_{5}
-8\hat{{\rm a}}_{6}-5\hat{{\rm a}}_{7}-12\hat{{\rm a}}_{8}-7\hat{{\rm S}}
-\hat{{\rm U}}_{2} -\hat{{\rm U}}_{3} -3\hat{{\rm U}}_{4} 
-2\hat{{\rm U}}_{5} -2\hat{{\rm U}}_{6}-7\hat{{\rm U}}_{7}-8\hat{{\rm U}}_{8} \\
{\rm E}_{14}  \ =& \  
8 {\rm F}^{{\rm a}} + 2 \hat{{\rm a}}_{2} - 2\hat{{\rm a}}_{4} -2\hat{{\rm a}}_{5}
-5\hat{{\rm a}}_{6}-3\hat{{\rm a}}_{7}-7\hat{{\rm a}}_{8}-4\hat{{\rm S}}
 -\hat{{\rm U}}_{3} -2\hat{{\rm U}}_{4} 
-\hat{{\rm U}}_{5} -2\hat{{\rm U}}_{6}-4\hat{{\rm U}}_{7}-5\hat{{\rm U}}_{8} \\
{\rm E}_{15}  \  =& \   
 {\rm F}^{{\rm a}}  - \hat{{\rm a}}_{4} -\hat{{\rm a}}_{5}
-2\hat{{\rm a}}_{6}-\hat{{\rm a}}_{7}-2\hat{{\rm a}}_{8}-\hat{{\rm S}}
-\hat{{\rm U}}_{7}-\hat{{\rm U}}_{8} \\
{\rm E}_{16}   \ =& \  \hat{{\rm a}}_{5} \\ 
{\rm E}_{23}  \ =& \  \hat{{\rm a}}_{8} \\
{\rm E}_{24}   \ =& \    
4 {\rm F}^{{\rm a}} +  \hat{{\rm a}}_{2} - \hat{{\rm a}}_{4} -\hat{{\rm a}}_{5}
-3\hat{{\rm a}}_{6}-2\hat{{\rm a}}_{7}-4\hat{{\rm a}}_{8}-2\hat{{\rm S}}
-\hat{{\rm U}}_{4} 
-\hat{{\rm U}}_{5} -\hat{{\rm U}}_{6}-2\hat{{\rm U}}_{7}-2\hat{{\rm U}}_{8} \\
{\rm E}_{25}  \  =& \  
9{\rm F}^{{\rm a}} + 2 \hat{{\rm a}}_{2} - 2\hat{{\rm a}}_{4} -2\hat{{\rm a}}_{5}
-6\hat{{\rm a}}_{6}-4\hat{{\rm a}}_{7}-9\hat{{\rm a}}_{8}-5\hat{{\rm S}}
-\hat{{\rm U}}_{2} -\hat{{\rm U}}_{3} -2\hat{{\rm U}}_{4} 
-\hat{{\rm U}}_{5} -2\hat{{\rm U}}_{6}-5\hat{{\rm U}}_{7}-6\hat{{\rm U}}_{8} \\
{\rm E}_{26}  \  =& \  
8 {\rm F}^{{\rm a}} + 2 \hat{{\rm a}}_{2} - 2\hat{{\rm a}}_{4} -2\hat{{\rm a}}_{5}
-6\hat{{\rm a}}_{6}-4\hat{{\rm a}}_{7}-9\hat{{\rm a}}_{8}-5\hat{{\rm S}}
-\hat{{\rm U}}_{3} -2\hat{{\rm U}}_{4} 
-\hat{{\rm U}}_{5} -\hat{{\rm U}}_{6}-5\hat{{\rm U}}_{7}-6\hat{{\rm U}}_{8} \\
{\rm E}_{34}  \  =& \  \hat{{\rm U}}_{7} \\
{\rm E}_{35} \ =& \    
  5{\rm F}^{{\rm a}} +  \hat{{\rm a}}_{2} - \hat{{\rm a}}_{4} -\hat{{\rm a}}_{5}
-3\hat{{\rm a}}_{6}-2\hat{{\rm a}}_{7}-5\hat{{\rm a}}_{8}-3\hat{{\rm S}}
-\hat{{\rm U}}_{3} -\hat{{\rm U}}_{4} 
-\hat{{\rm U}}_{5} -\hat{{\rm U}}_{6}-3\hat{{\rm U}}_{7}-3\hat{{\rm U}}_{8} \\
{\rm E}_{36} \  =& \    
  4{\rm F}^{{\rm a}} +  \hat{{\rm a}}_{2} - \hat{{\rm a}}_{4} -\hat{{\rm a}}_{5}
-3\hat{{\rm a}}_{6}-2\hat{{\rm a}}_{7}-5\hat{{\rm a}}_{8}-3\hat{{\rm S}}
-\hat{{\rm U}}_{4} 
 -\hat{{\rm U}}_{6}-3\hat{{\rm U}}_{7}-3\hat{{\rm U}}_{8} \\
{\rm E}_{45} \  =& \   {\rm F}^{{\rm a}}-\hat{{\rm U}}_{4} = \hat{{\rm V}}_{4} \\
{\rm E}_{46} \  =& \   \hat{{\rm U}}_{1} \\
{\rm E}_{56} \  =& \    
   5 {\rm F}^{{\rm a}} +  \hat{{\rm a}}_{2} - \hat{{\rm a}}_{4} -\hat{{\rm a}}_{5}
-3\hat{{\rm a}}_{6}-2\hat{{\rm a}}_{7}-5\hat{{\rm a}}_{8}-3\hat{{\rm S}}
-\hat{{\rm U}}_{4} 
-\hat{{\rm U}}_{5} -\hat{{\rm U}}_{6}-3\hat{{\rm U}}_{7}-4\hat{{\rm U}}_{8}
\end{align*}
We shall show that the classes 
\begin{equation}
\label{classesinquestion}
{\rm E}_{13}, \ {\rm E}_{14}, \ {\rm E}_{15} , \ {\rm E}_{24}, \ {\rm E}_{25}, \ {\rm E}_{26},  \ {\rm E}_{35}, \   {\rm E}_{36} , \ {\rm E}_{56}
\end{equation} 
are effective and their associated linear system consists of a single smooth rational curve. In order to accomplish this goal, 
we start a blow-down process 
\begin{equation}
\label{blowdownproc}
 {\rm R}={\rm R}_{15} \ \rightarrow \ {\rm R}_{14} \ \rightarrow  \ {\rm R}_{13} \ \rightarrow \ \cdots \ \rightarrow {\rm R}_{3} \ \rightarrow \ 
 {\rm R}_{2} \ \rightarrow  \ \widetilde{{\rm R}}_{1} 
\end{equation} 
by collapsing a sequence of exceptional rational curves in the fibers of the ruling $\varphi_{{\rm R}}$. The sequence of 
fourteen exceptional curves is as follows:
$$ \hat{{\rm E}}_{15}=\hat{{\rm U}}_{1}, \ \hat{{\rm E}}_{14}=\hat{{\rm U}}_{2}, \ \hat{{\rm E}}_{13}=\hat{{\rm U}}_{3}, \ 
\hat{{\rm E}}_{12}=\hat{{\rm V}}_{4}, \ 
\hat{{\rm E}}_{11}=\hat{{\rm V}}_{5}, \ \hat{{\rm E}}_{10}=\hat{{\rm V}}_{6}, \    
\hat{{\rm E}}_{9}=\hat{{\rm U}}_{7}, \  \hat{{\rm E}}_{8}=\hat{{\rm U}}_8,  $$
$$  \hat{{\rm E}}_{7}=\hat{{\rm a}}_{8}, \ \hat{{\rm E}}_{6}=\hat{{\rm a}}_{6}, \ \hat{{\rm E}}_{5}=\hat{{\rm S}} , \ 
\hat{{\rm E}}_{4}=\hat{{\rm a}}_{7} , \ 
\hat{{\rm E}}_{3}=\hat{{\rm a}}_{5}, \ \hat{\hat{{\rm E}}}_{2}=\hat{{\rm a}}_{3}.  $$
By a slight abuse, we keep the notation for the various rational curves involved, as they get pushed-forward under blow-downs. The resulting surface 
$\widetilde{{\rm R}}_{1} $ is smooth, rational and minimally ruled. Hence, by standard results on ruled surfaces (see, for instance, Chapter III of 
\cite{beauville}), the surface $\widetilde{{\rm R}}_{1} $ is isomorphic to one of 
the Hirzebruch surfaces $ \mathbb{F}_n$, $ n \geq 0$. The cohomology group ${\rm H}^2(\widetilde{{\rm R}}_{1}, \mathbb{Z})$ has 
rank two and is spanned by the classes of two rulings, one induced from $ \varphi_{{\rm R}}$ and having $\hat{{\rm a}}_{2}$ and $\hat{{\rm Q}}$ as fibers, 
and a second ruling with $\hat{{\rm a}}_{4}$ as fiber.   
$$
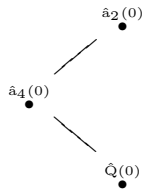
\def\objectstyle{\scriptstyle}
\def\labelstyle{\scriptstyle}
\xymatrix @-0.9pc
{
 &    \stackrel{\hat{{\rm a}}_{2}(0)}{\bullet} 
 & 
 & & & \\ 
 \stackrel{\hat{{\rm a}}_{4}(0)}{\bullet} \ar @{-} [ur] \ar @{-} [dr]  & &   & 
 \\
 &       \stackrel{\hat{{\rm Q}}(0)}{\bullet} 
 &  & & \\
 }
$$
It follows then that 
$\widetilde{{\rm R}}_{1} $ is isomorphic to $\mathbb{F}_0 =  \mathbb{P}^1 \times \mathbb{P}^1$. Moreover, if we remove the 
last blow-down in $ (\ref{blowdownproc})$ and instead we collapse the exceptional curves 
$$ \hat{{\rm E}}_{2} = \hat{{\rm a}}_2 \ {\rm and} \  \hat{{\rm E}}_{1}=\hat{{\rm a}}_{4} $$
$$ {\rm R}_{2} \ \rightarrow \ {\rm R}_{1} \ \rightarrow  \ {\rm R}_{0}, $$
the resulting surface $ {\rm R}_0 $ is a copy of the projective plane $\mathbb{P}^2$.
\par The blow-down construction determines the following configuration in $ {\rm R}_0 $. First, there are $x_1$, $x_2$ - the two 
distinct points of $ {\rm R}_0 $ where the last two exceptional curves $ \hat{{\rm a}}_4 $ 
and $\hat{{\rm a}}_{2}$ collapse. The push-forward of $\hat{{\rm a}}_2$ is the line $l$ joining $x_1$ and $x_2$. One has seven other 
distinct lines 
$$ l', \ l_1, \ l_2, \ l_3, \ l_4, \ l_5, \ l_6 $$ 
obtained as push-forward of 
$$ \widetilde{{\rm Q}}, \ \hat{{\rm V}}_{1}, \ \hat{{\rm V}}_{2}, \ \hat{{\rm V}}_{3}, \ 
\hat{{\rm U}}_{4}, \ \hat{{\rm U}}_{5}, \ \hat{{\rm U}}_{6}. $$
The line $l'$ passes through $x_1$ but not $x_2$. The six lines $ l_1, l_2, \cdots  l_6 $ meet at $x_2$ but do not pass through $x_1$. Denote by 
$ y_i$ with $ 1 \leq i \leq 6$, the six points of intersection between the lines $l'$ and $l_i$, respectively. Going backwards through the blow-down 
process $(\ref{blowdownproc})$, one recovers the rational surface ${\rm R}$ as the blow-up of the projective plane ${\rm R}_0$ at a sequence 
of fifteen points:
$$ p_1, \ p_2 , \ p_3 , \ \cdots \ p_{15} $$   
where $p_1=x_1$, $p_2=x_2$, the points $ p_3, p_4, \cdots p_9 $ are infinitely near $x_1$ with $p_3$ representing the tangent direction of $l'$, 
the points $p_{10}, p_{11}, p_{12}$ are infinitely near $x_2$ and represent the tangent directions of $l_6, l_5, l_4$, and 
$ p_{13}=y_3$, $p_{14}=y_2$, $p_{15}=y_1$. 
\par Let $ \hat{{\rm H}}$ be the class of a hyperplane section in  ${\rm R}_0 $ and 
denote by $\hat{{\rm E}}_{1}, \ \hat{{\rm E}}_{2}, \ \cdots \ \hat{{\rm E}}_{15} $ the strict transforms of the fifteen exceptional curves 
associated to the blow-up $ {\rm R} \rightarrow {\rm R}_0$. The sixteen classes 
$ \hat{{\rm H}}, \hat{{\rm E}}_{1}, \ \hat{{\rm E}}_{2}, \ \cdots \ \hat{{\rm E}}_{15} $ form a basis over the integers for ${\rm H}^2({\rm R},\mathbb{Z}) $ 
and, with respect to this basis, the classes of $(\ref{classesinquestion})$ are as follows:  
\begin{align*}
{\rm E}_{13} \ =& \ 5\hat{{\rm H}}- 3\hat{\rm E}_1 - 2 \left ( \hat{\rm E}_2+ \hat{\rm E}_4 + \hat{\rm E}_5 \right ) - 
\left ( \hat{\rm E}_8 + \hat{\rm E}_{10} + \hat{\rm E}_{11} + \hat{\rm E}_{13} + \hat{\rm E}_{14} \right )  \\
{\rm E}_{14} \ =& \ 3\hat{{\rm H}}- 2\hat{\rm E}_1 - \left ( \hat{\rm E}_2+ \hat{\rm E}_4 + \hat{\rm E}_5 + \hat{\rm E}_8 + 
\hat{\rm E}_{11} + \hat{\rm E}_{13}  \right )  \\
{\rm E}_{15} \ =& \ \hat{{\rm H}}- \hat{\rm E}_1 - \hat{\rm E}_2   \\
{\rm E}_{24} \ =& \ \hat{{\rm H}}- \hat{\rm E}_1 - \hat{\rm E}_4   \\
{\rm E}_{25} \ =& \ 4\hat{{\rm H}}- 2 \left ( \hat{\rm E}_1 + \hat{\rm E}_2+ \hat{\rm E}_4  \right ) - 
\left ( \hat{\rm E}_5 + \hat{\rm E}_8 + \hat{\rm E}_{11} + \hat{\rm E}_{13} + \hat{\rm E}_{14} \right )  \\
{\rm E}_{26} \ =& \ 4\hat{{\rm H}}- 2 \left ( \hat{\rm E}_1 + \hat{\rm E}_2+ \hat{\rm E}_4  \right ) - 
\left ( \hat{\rm E}_5 + \hat{\rm E}_8 + \hat{\rm E}_{10} + \hat{\rm E}_{11} + \hat{\rm E}_{13} \right )  \\
{\rm E}_{35} \ =& \ 2\hat{{\rm H}}- \left ( \hat{\rm E}_1+ \hat{\rm E}_2 + \hat{\rm E}_4 + \hat{\rm E}_5 + \hat{\rm E}_{13}  \right )  \\
{\rm E}_{36} \ =& \ 2\hat{{\rm H}}- \left ( \hat{\rm E}_1+ \hat{\rm E}_2 + \hat{\rm E}_4 + \hat{\rm E}_5 + \hat{\rm E}_{11}  \right )  \\
{\rm E}_{56} \ =& \ 2\hat{{\rm H}}- \left ( \hat{\rm E}_1+ \hat{\rm E}_2 + \hat{\rm E}_4 + \hat{\rm E}_5 + \hat{\rm E}_{8}  \right ).  
\end{align*}
Moreover, the class of the fiber of the ruling $\varphi_{{\rm R}}$ is 
$$ {\rm F}^{{\rm a}} \ = \ \hat{{\rm H}}- \hat{\rm E}_2.   $$
One verifies then that the nine points $p_1, p_2, p_4, p_5, p_8, p_{10}, p_{11}, p_{13}, p_{14} $ are in a general enough position that
all above classes are effective and each is represented by a unique smooth rational curve. Abusing the notation, we denote these rational curves of 
${\rm R}$ by same symbol as their cohomology class. 
\par We have obtained fifteen disjoint rational curves on ${\rm R}$, denoted ${\rm E}_{ij}$ with $ 1 \leq i < j \leq 6$. All curves ${\rm E}_{ij}$ have 
self-intersection $-1$. By blowing down ${\rm E}_{ij}$, one obtains another copy of the projective plane $\mathbb{P}^2$. Denote by $q_{ij}$ 
the fifteen distinct points obtained by collapsing the exceptional curves. The push-forwards of the 
six curves:
$$ \hat{{\rm a}}_{4}, \ 
\hat{{\rm a}}_{7},  \  \hat{{\rm S}}, \ \hat{{\rm K}}, \ \hat{{\rm a}}_{2}, \ \hat{{\rm Q}}  $$ 
form a configuration $\mathcal{L}= \{ {\rm L}_1, {\rm L}_2, \cdots {\rm L}_6 \}$ of six lines in this projective plane, meeting 
at the fifteen points $q_{ij}$. The push-forward of $\hat{{\rm U}}_{7}$ 
is a conic passing through the five points $q_{13},q_{14}, q_{25}, q_{26}, q_{56}$ but this conic does not contain $q_{34}$. Therefore, 
the six-line configuration $\mathcal{L}$ is non-Kummer.  
\par A slight modification of the above arguments gives a proof for the case of a generic special polarized pair $({\rm X}, i)$. One obtains 
the fifteen disjoint rational curves ${\rm E}_{ij}$ in the same manner as above. Then one checks that the conic through 
$q_{13},q_{14}, q_{25}, q_{26}, q_{56}$ also contains $q_{34}$. This fact, in turn, implies the existence of a rational curve 
${\rm E}_{\varnothing}$ tangent to all the six lines of the configuration $\mathcal{L}$.   
\end{proof}

\end{document}